\documentclass[11pt,reqno]{amsart}
\usepackage[margin=1.15in]{geometry}
\usepackage{amsmath,amssymb,amsthm}
\usepackage{booktabs}
\usepackage{tikz}
\usepackage[colorlinks,linkcolor=blue!55!black,citecolor=blue!55!black,urlcolor=blue!55!black]{hyperref}

\definecolor{codec}{RGB}{31,94,168}
\definecolor{newc}{RGB}{214,108,12}
\definecolor{delc}{RGB}{190,30,30}
\colorlet{gridc}{gray!40}
\colorlet{vtx}{gray!65}
\colorlet{seam}{black!65}
\definecolor{pc1}{RGB}{198,219,239}
\definecolor{pc2}{RGB}{199,233,192}
\definecolor{pc3}{RGB}{253,208,162}
\definecolor{pc4}{RGB}{218,218,235}
\definecolor{pc5}{RGB}{252,187,161}

\newtheorem{theorem}{Theorem}
\newtheorem{lemma}[theorem]{Lemma}
\newtheorem{proposition}[theorem]{Proposition}
\newtheorem{conjecture}[theorem]{Conjecture}
\newtheorem{corollary}[theorem]{Corollary}
\theoremstyle{definition}
\newtheorem{example}[theorem]{Example}
\theoremstyle{remark}
\newtheorem{remark}[theorem]{Remark}

\newcommand{\Z}{\mathbb{Z}}

\title[An improved lower bound for $\rho(F_2(C_n))$]{An improved lower bound for the packing number\\ of the 2-token graph of the cycle}

\author{Luis Manuel Rivera}
\address{Unidad Acad\'emica de Matem\'aticas, Universidad Aut\'onoma de Zacatecas, Zacatecas, M\'exico}
\email{luismanuel.rivera@gmail.com}


\begin{document}

\begin{abstract}\sloppy
Let $F_2(C_n)$ be the $2$-token graph of the cycle $C_n$ and let $\rho$ denote the packing number. Gómez Soto and Ríos-Castro recently proved that $\rho(F_2(C_n))\ge a(n)$ for $n\ge 19$, where $a(n)$ is an explicit expression. In this note, we prove that
\[
\rho(F_2(C_n))\ \ge\ \left\lfloor\frac{n(n-2)}{10}\right\rfloor+1\qquad\text{for every } n\ge 3,
\]
which improves $a(n)$ by one whenever $n\equiv 0,2\pmod{10}$. 
\end{abstract}

\maketitle

\section{Introduction}

For a simple graph $G$ and an integer $k\ge 1$, the \emph{$k$-token graph} $F_k(G)$ is the graph whose vertices are all the $k$-subsets of $V(G)$, where two vertices are adjacent whenever their symmetric difference is an edge of $G$ (see, e.g.,~\cite{FFHHUW}). There is an equivalent, more intuitive way to work with token graphs: a vertex of $F_k(G)$ is a configuration of $k$ indistinguishable tokens on distinct vertices of $G$, and two configurations are adjacent when one is obtained from the other by moving one token to an adjacent unoccupied vertex. For $k=2$, these graphs are also known as \emph{double vertex graphs}~\cite{ABEL}. There is growing interest in the study of token graphs; see e.g.,~\cite{ATLR, DFM, RFSGGL, GLRR}. A set $P\subseteq V(G)$ is a \emph{packing set} if $d_G(x,y)\ge 3$ for all distinct $x,y\in P$, and the \emph{packing number} $\rho(G)$ is the maximum size of a packing set.

A constant-weight-$2$ code correcting a single adjacent transposition is exactly a packing set of $F_2(P_n)$; this problem goes back to Sloane~\cite{S}, and $\rho(F_2(P_n))$ was determined in~\cite{GLRR}, confirming a conjecture recorded in the OEIS~\cite{OEIS}. If the first and the last bit are also regarded as adjacent, one is led to $F_2(C_n)$. Gómez Soto and R\'ios-Castro~\cite{GR} obtained the first general lower bound for its packing number.

\begin{theorem}\label{theoremGR}
For every $n\ge 19$, $\rho(F_2(C_n))\ge a(n)$, where
\[
a(n)=\left\lceil \frac{\binom{n-1}{2}-n+1-2\lfloor\frac{n-1}{2}\rfloor}{5}\right\rceil+\left\lfloor\frac{n-1}{2}\right\rfloor .
\]
\end{theorem}

\subsection{Main results}

The main result in this note is the following:

\begin{theorem}\label{maintheorem}
For every $n\ge 3$,
\[
\rho(F_2(C_n))\ \ge\ \left\lfloor\frac{n(n-2)}{10}\right\rfloor+1 .
\]
\end{theorem}

To compare Theorem~\ref{maintheorem} with Theorem~\ref{theoremGR}, we simplify $a(n)$ using the following elementary identity.

\begin{proposition}\label{prop:a}
For every $n\ge 2$, $\ a(n)=\lceil n(n-2)/10\rceil$.
\end{proposition}

\begin{proof}
Put $m=n-1$ and $h=\lfloor m/2\rfloor$. Since $\binom{m}{2}-m=m(m-3)/2$ and $h\in\Z$, we have $a(n)=\bigl\lceil \bigl(m(m-3)/2+3h\bigr)/5\bigr\rceil$, that is, $a(n)=\lceil m^2/10\rceil$ if $m$ is even and $a(n)=\lceil (m^2-3)/10\rceil$ if $m$ is odd. On the other hand, $n(n-2)=m^2-1$. If $m$ is even, then $m^2\equiv 0,4,6\pmod{10}$, so $\lceil m^2/10\rceil=\lceil (m^2-1)/10\rceil$. If $m$ is odd, then $m^2-1\equiv 0,4,8\pmod{10}$, so $\lceil (m^2-3)/10\rceil=\lceil (m^2-1)/10\rceil$.
\end{proof}

Since $n(n-2)/10$ is an integer exactly when $n\equiv 0,2\pmod{10}$, the bound of Theorem~\ref{maintheorem} equals $a(n)$ when $n\not\equiv 0,2\pmod{10}$ and equals $a(n)+1$ when $n\equiv 0,2\pmod{10}$; moreover, it holds for every $n\ge 3$. The packing sets are explicit: for $n\not\equiv 0,2\pmod{10}$ one of the sets $S_0,S_1$ of Section~\ref{sec:S} attains the bound, and for $n\equiv 0,2\pmod{10}$ the set $S^+$ of Section~\ref{sec:plus} does. For comparison, a vertex $\{i,i+1\}$ of $F_2(C_n)$ has degree $2$ and every other vertex has degree $4$, so the disjointness of closed neighbourhoods gives $\rho(F_2(C_n))\le \frac15\binom n2+O(n)$; the bound of Theorem~\ref{maintheorem} is $\frac15\binom n2-\frac n{10}+O(1)$. Using the values of $\rho(F_2(C_n))$ for $n\le 7$ from \cite[Table~1]{GR} and computing the remaining ones by integer programming with the solver HiGHS, we have checked that the bound of Theorem~\ref{maintheorem} is attained for every $3\le n\le 32$ except $n=8$ and $n=11$, where $\rho(F_2(C_8))=6$ and $\rho(F_2(C_{11}))=11$. Our values agree with \cite[Table~1]{GR} for $n\le 19$. We pose the following conjecture.

\begin{conjecture}\label{conj:exact}
For every $n\ge 12$,
\[
\rho(F_2(C_n)) = \left\lfloor\frac{n(n-2)}{10}\right\rfloor + 1.
\]
\end{conjecture}
By Theorem~\ref{maintheorem}, Conjecture~\ref{conj:exact} is equivalent to the upper bound $\rho(F_2(C_n))\le\lfloor n(n-2)/10\rfloor+1$ for $n\ge 12$.

The key observation is that the rhombus-shaped region used in~\cite{GR} is a fundamental domain of a glide reflection; making this explicit gives a short proof valid for all $n\ge 3$ and shows where one extra vertex can be gained. The note is organised as follows. Section~\ref{sec:model} describes $F_2(C_n)$ as the quotient of a strip of $\Z^2$ by a glide reflection, and Section~\ref{sec:proof} contains the proof of Theorem~\ref{maintheorem}.


\section{\texorpdfstring{$F_2(C_n)$}{F2(Cn)} as a discrete Möbius band}\label{sec:model}

From now on, $n\ge 3$ is an integer and we label the vertices of $C_n$ by the elements of $\Z_n=\{0,1,\dots,n-1\}$, so that $i$ and $j$ are adjacent if and only if $j\equiv i\pm1\pmod n$. We work in the set
\[
\Sigma=\{(s,g)\in\Z^2:\ 1\le g\le n-1,\ s\equiv g\pmod 2\},
\]
with the norm $\|(s,g)\|=\max\{|s|,|g|\}$, and we write $s(z)$ and $g(z)$ for the coordinates of $z\in\Z^2$. We define
\begin{align*}
\pi(s,g)&=\Bigl\{\tfrac{s-g}{2}\bmod n,\ \tfrac{s+g}{2}\bmod n\Bigr\},\\
\sigma(s,g)&=(s+n,\,n-g), \\
D&=\{z\in\Sigma:\ 0\le s(z)\le n-1\}.
\end{align*}
Thus $\pi(z)$ is the vertex of $F_2(C_n)$ whose tokens are at $(s-g)/2$ and $(s+g)/2$: the coordinate $s$ is the sum of the positions of the tokens and $g$ is their gap. The map $\sigma$ is a glide reflection with $\sigma(\Sigma)=\Sigma$ and $\pi\circ\sigma=\pi$; see Figure~\ref{fig:glide}.

\begin{lemma}\label{lem:model}
\begin{enumerate}
\item[(i)] For $z,z'\in\Sigma$ we have $\pi(z)=\pi(z')$ if and only if $z'=\sigma^k(z)$ for some $k\in\Z$. In particular, $\pi$ restricts to a bijection from $D$ onto $V(F_2(C_n))$; for $0\le a<b\le n-1$ the preimage of $\{a,b\}$ in $D$ is $(a+b,\,b-a)$ if $a+b\le n-1$, and $(a+b-n,\,n-b+a)$ if $a+b\ge n$.
\item[(ii)] For all $z,z'\in\Sigma$,
\[
d_{F_2(C_n)}\bigl(\pi(z),\pi(z')\bigr)\ \ge\ \min_{k\in\Z}\ \|z-\sigma^k(z')\| .
\]
\end{enumerate}
\end{lemma}

\begin{proof}
The change of variables $u=(s-g)/2$, $v=(s+g)/2$ is a bijection from $\Sigma$ onto the set $\{(u,v)\in\Z^2:0<v-u<n\}$. In these coordinates $\pi(u,v)=\{u\bmod n,\,v\bmod n\}$, which is a $2$-subset of $\Z_n$ because $0<v-u<n$, and $\sigma(u,v)=(v,\,u+n)$; in particular $\pi\circ\sigma=\pi$. Moreover, $\sigma^2(u,v)=(u+n,\,v+n)$, so for every $m\in\Z$ and every $(a,b)\in\Z^2$,
\[
\sigma^{2m}(a,b)=(a+mn,\ b+mn)\qquad\text{and}\qquad \sigma^{2m+1}(a,b)=(b+mn,\ a+(m+1)n).
\]

(i) Since $\pi\circ\sigma=\pi$, it suffices to show that every fibre of $\pi$ is a single $\sigma$-orbit. Fix $0\le a<b\le n-1$ and let $(u,v)$ be a point with $0<v-u<n$ and $\pi(u,v)=\{a,b\}$. If $u\equiv a$ and $v\equiv b\pmod n$, write $(u,v)=(a+mn,\,b+m'n)$; then $v-u=b-a+(m'-m)n$, and since $0<b-a<n$, the condition $0<v-u<n$ forces $m'=m$, so $(u,v)=\sigma^{2m}(a,b)$. Otherwise $u\equiv b$ and $v\equiv a\pmod n$; writing $(u,v)=(b+mn,\,a+m'n)$ we get $v-u=a-b+(m'-m)n$, and since $-n<a-b<0$ this forces $m'=m+1$, so $(u,v)=\sigma^{2m+1}(a,b)$. Thus the fibre of $\{a,b\}$ is exactly the orbit $\{\sigma^k(a,b):k\in\Z\}$. Since $s(\sigma^k z)=s(z)+kn$ and $s(a,b)=a+b$, the points of this orbit have first coordinate $a+b+kn$, $k\in\Z$, and exactly one of these values lies in $[0,n-1]$. Hence exactly one point of the orbit lies in $D$: it is $(a,b)$ if $a+b\le n-1$, and $\sigma^{-1}(a,b)=(b-n,\,a)$ if $a+b\ge n$. Written in the coordinates $(s,g)=(u+v,\,v-u)$, these points are $(a+b,\,b-a)$ and $(a+b-n,\,n-b+a)$, respectively.

(ii) Let $d=d_{F_2(C_n)}(\pi(z),\pi(z'))$, let $\pi(z)=A_0,A_1,\dots,A_d=\pi(z')$ be a geodesic, and put $z_0=z$. We construct inductively points $z_1,\dots,z_d\in\Sigma$ with $\pi(z_i)=A_i$ and $\|z_i-z_{i-1}\|=1$. If $z_{i-1}=(u,v)$ satisfies $\pi(z_{i-1})=A_{i-1}$, then $A_i$ is obtained from $A_{i-1}$ by moving the token at $u\bmod n$ or the one at $v\bmod n$ by $\pm1$, and we let $z_i=(u\pm1,\,v)$ or $z_i=(u,\,v\pm1)$ accordingly. Then $\pi(z_i)=A_i$ and $v-u$ changes by $\pm1$. The new value lies in $[0,n]$ and is neither $0$ nor $n$ because the two tokens of $A_i$ are distinct, so $z_i\in\Sigma$. In the coordinates $(s,g)$ each step is one of $(\pm1,\pm1)$, so $\|z_i-z_{i-1}\|=1$. Hence $\|z-z_d\|\le d$. Since $\pi(z_d)=\pi(z')$, part (i) gives $z_d=\sigma^k(z')$ for some $k\in\Z$, and therefore $\min_{k\in\Z}\|z-\sigma^k(z')\|\le d$.
\end{proof}

\begin{figure}[ht]
\centering
\begin{tikzpicture}
\fill[gray!10] (-2.300,0.115) rectangle (4.600,2.185);
\draw[thick] (-2.300,0.115)--(4.600,0.115);
\draw[thick] (-2.300,2.185)--(4.600,2.185);
\fill[codec!25] (-0.115,0.115) rectangle (2.185,2.185);
\node at (1.035,1.150) {$D$};
\node at (3.335,1.150) {$\sigma(D)$};
\node at (-1.265,1.150) {$\sigma^{-1}(D)$};
\draw[dashed,gray] (-2.415,0.115)--(-2.415,2.185);
\draw[dashed,gray] (-0.115,0.115)--(-0.115,2.185);
\draw[dashed,gray] (2.185,0.115)--(2.185,2.185);
\draw[dashed,gray] (4.485,0.115)--(4.485,2.185);
\draw[newc,very thick,fill=newc!40] (0.345,0.345)--(0.345,0.851)--(0.621,0.759)--(0.345,0.667)--cycle;
\draw[newc,very thick,fill=newc!40] (2.645,1.955)--(2.645,1.449)--(2.921,1.541)--(2.645,1.633)--cycle;
\draw[newc,very thick,fill=newc!40] (-1.955,1.955)--(-1.955,1.449)--(-1.679,1.541)--(-1.955,1.633)--cycle;
\draw[->,thick,bend left=25] (0.575,2.369) to node[above]{\small $\sigma(s,g)=(s+n,\,n-g)$} (2.875,2.369);
\node[right] at (4.600,0.115) {\scriptsize $g=0$};
\node[right] at (4.600,2.185) {\scriptsize $g=n$};
\end{tikzpicture}
\caption{The strip $\Sigma$ in the coordinates $(s,g)$, the fundamental domain $D$ and its images under the glide reflection $\sigma$. The flag shows that $\sigma$ reverses the orientation: $F_2(C_n)$ is a Möbius band.}
\label{fig:glide}
\end{figure}
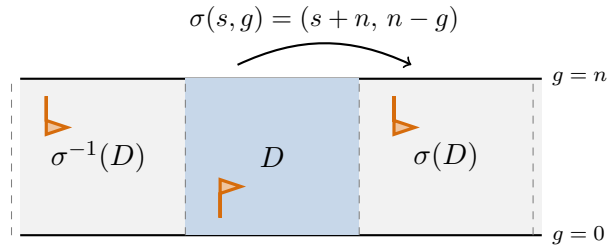

Lemma~\ref{lem:model} says that $F_2(C_n)$ is a discrete Möbius band: the strip $\Sigma$ modulo the glide reflection $\sigma$. Its boundary, the row $g=1$ (identified with the row $g=n-1$), consists of the vertices $\{i,i+1\}$.

For an even integer $c$ let
\[
L_c=\{(s,g)\in\Z^2:\ g+3s\equiv c\pmod{10}\}.
\]
Since $c$ is even, every point of $L_c$ satisfies $s\equiv g\pmod 2$. Since $u+2v=(3s+g)/2$, in the coordinates $(u,v)$ the set $L_c$ is $\{u+2v\equiv c/2\pmod 5\}$, a perfect code of $\Z^2$ in the $\ell_1$ metric \cite{GW}; see Figure~\ref{fig:perfect}. We only need the packing property.

\begin{remark}\label{rem:GR}
In \cite{GR}, the graph $F_2(C_n)$ is represented by the graph $T_C(n-1)$, whose vertex $(x,y)$, $1\le x\le y\le n-1$, corresponds to the pair of tokens $\{x,y+1\}\subseteq\{1,\dots,n\}$, and the packing sets are built from the classes $f^{-1}(\ell)$ of the function $f(x,y)=x+2y\pmod 5$. With our labelling $V(C_n)=\Z_n$, these coordinates are
\[
(x,y)=(u+1,\,v)=\Bigl(\tfrac{s-g}{2}+1,\ \tfrac{s+g}{2}\Bigr),
\qquad\text{and}\qquad
L_c=\bigl\{(s,g):\ f(x,y)\equiv c/2+1\pmod 5\bigr\}.
\]
Thus the sets $L_c$ are exactly the classes $f^{-1}(\ell)$ of \cite{GR}, with $\ell\equiv c/2+1\pmod 5$, written in the rotated coordinates $(s,g)$. In these coordinates the $\ell_1$ metric used in \cite{GR} becomes the norm $\|\cdot\|$, since $|\Delta x|+|\Delta y|=\max\{|\Delta s|,|\Delta g|\}$, and the extra edges of $T_C(n-1)$ are accounted for by the glide reflection $\sigma$. The domain $D$ plays a role analogous to that of the rhombus graphs $\lozenge_{n,i}$ of \cite{GR}.
\end{remark}

\begin{lemma}\label{lem:code}
If $z\neq z'$ are points of $L_c$, then $\|z-z'\|\ge 3$.
\end{lemma}

\begin{proof}
Let $(\delta_s,\delta_g)=z-z'$. Then $\delta_g+3\delta_s\equiv 0\pmod{10}$ and $\delta_s\equiv\delta_g\pmod 2$. Up to sign, the nonzero vectors with $\delta_s\equiv\delta_g\pmod 2$ and norm at most $2$ are $(1,1)$, $(1,-1)$, $(2,0)$, $(0,2)$, $(2,2)$, $(2,-2)$, for which $\delta_g+3\delta_s$ equals $4,2,6,2,8,4$. None of these is divisible by $10$.
\end{proof}

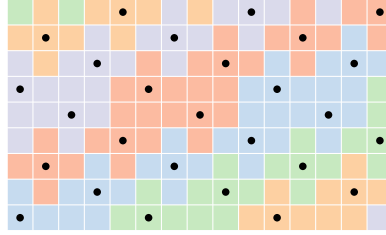
\begin{figure}[ht]
\centering
\begin{tikzpicture}
\clip (-0.5*0.34,-0.5*0.34) rectangle (4.930,2.890);
\fill[pc5] (-0.850,0.170) rectangle (-0.510,0.510);
\fill[pc5] (-0.510,0.170) rectangle (-0.170,0.510);
\fill[pc5] (-1.190,0.170) rectangle (-0.850,0.510);
\fill[pc5] (-0.850,0.510) rectangle (-0.510,0.850);
\fill[pc5] (-0.850,-0.170) rectangle (-0.510,0.170);
\fill[pc3] (-0.850,1.870) rectangle (-0.510,2.210);
\fill[pc3] (-0.510,1.870) rectangle (-0.170,2.210);
\fill[pc3] (-1.190,1.870) rectangle (-0.850,2.210);
\fill[pc3] (-0.850,2.210) rectangle (-0.510,2.550);
\fill[pc3] (-0.850,1.530) rectangle (-0.510,1.870);
\fill[pc1] (-0.510,-0.850) rectangle (-0.170,-0.510);
\fill[pc1] (-0.170,-0.850) rectangle (0.170,-0.510);
\fill[pc1] (-0.850,-0.850) rectangle (-0.510,-0.510);
\fill[pc1] (-0.510,-0.510) rectangle (-0.170,-0.170);
\fill[pc1] (-0.510,-1.190) rectangle (-0.170,-0.850);
\fill[pc4] (-0.510,0.850) rectangle (-0.170,1.190);
\fill[pc4] (-0.170,0.850) rectangle (0.170,1.190);
\fill[pc4] (-0.850,0.850) rectangle (-0.510,1.190);
\fill[pc4] (-0.510,1.190) rectangle (-0.170,1.530);
\fill[pc4] (-0.510,0.510) rectangle (-0.170,0.850);
\fill[pc2] (-0.510,2.550) rectangle (-0.170,2.890);
\fill[pc2] (-0.170,2.550) rectangle (0.170,2.890);
\fill[pc2] (-0.850,2.550) rectangle (-0.510,2.890);
\fill[pc2] (-0.510,2.890) rectangle (-0.170,3.230);
\fill[pc2] (-0.510,2.210) rectangle (-0.170,2.550);
\fill[pc1] (-0.170,-0.170) rectangle (0.170,0.170);
\fill[pc1] (0.170,-0.170) rectangle (0.510,0.170);
\fill[pc1] (-0.510,-0.170) rectangle (-0.170,0.170);
\fill[pc1] (-0.170,0.170) rectangle (0.170,0.510);
\fill[pc1] (-0.170,-0.510) rectangle (0.170,-0.170);
\fill[pc4] (-0.170,1.530) rectangle (0.170,1.870);
\fill[pc4] (0.170,1.530) rectangle (0.510,1.870);
\fill[pc4] (-0.510,1.530) rectangle (-0.170,1.870);
\fill[pc4] (-0.170,1.870) rectangle (0.170,2.210);
\fill[pc4] (-0.170,1.190) rectangle (0.170,1.530);
\fill[pc2] (-0.170,3.230) rectangle (0.170,3.570);
\fill[pc2] (0.170,3.230) rectangle (0.510,3.570);
\fill[pc2] (-0.510,3.230) rectangle (-0.170,3.570);
\fill[pc2] (-0.170,3.570) rectangle (0.170,3.910);
\fill[pc2] (-0.170,2.890) rectangle (0.170,3.230);
\fill[pc5] (0.170,0.510) rectangle (0.510,0.850);
\fill[pc5] (0.510,0.510) rectangle (0.850,0.850);
\fill[pc5] (-0.170,0.510) rectangle (0.170,0.850);
\fill[pc5] (0.170,0.850) rectangle (0.510,1.190);
\fill[pc5] (0.170,0.170) rectangle (0.510,0.510);
\fill[pc3] (0.170,2.210) rectangle (0.510,2.550);
\fill[pc3] (0.510,2.210) rectangle (0.850,2.550);
\fill[pc3] (-0.170,2.210) rectangle (0.170,2.550);
\fill[pc3] (0.170,2.550) rectangle (0.510,2.890);
\fill[pc3] (0.170,1.870) rectangle (0.510,2.210);
\fill[pc1] (0.510,-0.510) rectangle (0.850,-0.170);
\fill[pc1] (0.850,-0.510) rectangle (1.190,-0.170);
\fill[pc1] (0.170,-0.510) rectangle (0.510,-0.170);
\fill[pc1] (0.510,-0.170) rectangle (0.850,0.170);
\fill[pc1] (0.510,-0.850) rectangle (0.850,-0.510);
\fill[pc4] (0.510,1.190) rectangle (0.850,1.530);
\fill[pc4] (0.850,1.190) rectangle (1.190,1.530);
\fill[pc4] (0.170,1.190) rectangle (0.510,1.530);
\fill[pc4] (0.510,1.530) rectangle (0.850,1.870);
\fill[pc4] (0.510,0.850) rectangle (0.850,1.190);
\fill[pc2] (0.510,2.890) rectangle (0.850,3.230);
\fill[pc2] (0.850,2.890) rectangle (1.190,3.230);
\fill[pc2] (0.170,2.890) rectangle (0.510,3.230);
\fill[pc2] (0.510,3.230) rectangle (0.850,3.570);
\fill[pc2] (0.510,2.550) rectangle (0.850,2.890);
\fill[pc1] (0.850,0.170) rectangle (1.190,0.510);
\fill[pc1] (1.190,0.170) rectangle (1.530,0.510);
\fill[pc1] (0.510,0.170) rectangle (0.850,0.510);
\fill[pc1] (0.850,0.510) rectangle (1.190,0.850);
\fill[pc1] (0.850,-0.170) rectangle (1.190,0.170);
\fill[pc4] (0.850,1.870) rectangle (1.190,2.210);
\fill[pc4] (1.190,1.870) rectangle (1.530,2.210);
\fill[pc4] (0.510,1.870) rectangle (0.850,2.210);
\fill[pc4] (0.850,2.210) rectangle (1.190,2.550);
\fill[pc4] (0.850,1.530) rectangle (1.190,1.870);
\fill[pc2] (1.190,-0.850) rectangle (1.530,-0.510);
\fill[pc2] (1.530,-0.850) rectangle (1.870,-0.510);
\fill[pc2] (0.850,-0.850) rectangle (1.190,-0.510);
\fill[pc2] (1.190,-0.510) rectangle (1.530,-0.170);
\fill[pc2] (1.190,-1.190) rectangle (1.530,-0.850);
\fill[pc5] (1.190,0.850) rectangle (1.530,1.190);
\fill[pc5] (1.530,0.850) rectangle (1.870,1.190);
\fill[pc5] (0.850,0.850) rectangle (1.190,1.190);
\fill[pc5] (1.190,1.190) rectangle (1.530,1.530);
\fill[pc5] (1.190,0.510) rectangle (1.530,0.850);
\fill[pc3] (1.190,2.550) rectangle (1.530,2.890);
\fill[pc3] (1.530,2.550) rectangle (1.870,2.890);
\fill[pc3] (0.850,2.550) rectangle (1.190,2.890);
\fill[pc3] (1.190,2.890) rectangle (1.530,3.230);
\fill[pc3] (1.190,2.210) rectangle (1.530,2.550);
\fill[pc2] (1.530,-0.170) rectangle (1.870,0.170);
\fill[pc2] (1.870,-0.170) rectangle (2.210,0.170);
\fill[pc2] (1.190,-0.170) rectangle (1.530,0.170);
\fill[pc2] (1.530,0.170) rectangle (1.870,0.510);
\fill[pc2] (1.530,-0.510) rectangle (1.870,-0.170);
\fill[pc5] (1.530,1.530) rectangle (1.870,1.870);
\fill[pc5] (1.870,1.530) rectangle (2.210,1.870);
\fill[pc5] (1.190,1.530) rectangle (1.530,1.870);
\fill[pc5] (1.530,1.870) rectangle (1.870,2.210);
\fill[pc5] (1.530,1.190) rectangle (1.870,1.530);
\fill[pc3] (1.530,3.230) rectangle (1.870,3.570);
\fill[pc3] (1.870,3.230) rectangle (2.210,3.570);
\fill[pc3] (1.190,3.230) rectangle (1.530,3.570);
\fill[pc3] (1.530,3.570) rectangle (1.870,3.910);
\fill[pc3] (1.530,2.890) rectangle (1.870,3.230);
\fill[pc1] (1.870,0.510) rectangle (2.210,0.850);
\fill[pc1] (2.210,0.510) rectangle (2.550,0.850);
\fill[pc1] (1.530,0.510) rectangle (1.870,0.850);
\fill[pc1] (1.870,0.850) rectangle (2.210,1.190);
\fill[pc1] (1.870,0.170) rectangle (2.210,0.510);
\fill[pc4] (1.870,2.210) rectangle (2.210,2.550);
\fill[pc4] (2.210,2.210) rectangle (2.550,2.550);
\fill[pc4] (1.530,2.210) rectangle (1.870,2.550);
\fill[pc4] (1.870,2.550) rectangle (2.210,2.890);
\fill[pc4] (1.870,1.870) rectangle (2.210,2.210);
\fill[pc2] (2.210,-0.510) rectangle (2.550,-0.170);
\fill[pc2] (2.550,-0.510) rectangle (2.890,-0.170);
\fill[pc2] (1.870,-0.510) rectangle (2.210,-0.170);
\fill[pc2] (2.210,-0.170) rectangle (2.550,0.170);
\fill[pc2] (2.210,-0.850) rectangle (2.550,-0.510);
\fill[pc5] (2.210,1.190) rectangle (2.550,1.530);
\fill[pc5] (2.550,1.190) rectangle (2.890,1.530);
\fill[pc5] (1.870,1.190) rectangle (2.210,1.530);
\fill[pc5] (2.210,1.530) rectangle (2.550,1.870);
\fill[pc5] (2.210,0.850) rectangle (2.550,1.190);
\fill[pc3] (2.210,2.890) rectangle (2.550,3.230);
\fill[pc3] (2.550,2.890) rectangle (2.890,3.230);
\fill[pc3] (1.870,2.890) rectangle (2.210,3.230);
\fill[pc3] (2.210,3.230) rectangle (2.550,3.570);
\fill[pc3] (2.210,2.550) rectangle (2.550,2.890);
\fill[pc2] (2.550,0.170) rectangle (2.890,0.510);
\fill[pc2] (2.890,0.170) rectangle (3.230,0.510);
\fill[pc2] (2.210,0.170) rectangle (2.550,0.510);
\fill[pc2] (2.550,0.510) rectangle (2.890,0.850);
\fill[pc2] (2.550,-0.170) rectangle (2.890,0.170);
\fill[pc5] (2.550,1.870) rectangle (2.890,2.210);
\fill[pc5] (2.890,1.870) rectangle (3.230,2.210);
\fill[pc5] (2.210,1.870) rectangle (2.550,2.210);
\fill[pc5] (2.550,2.210) rectangle (2.890,2.550);
\fill[pc5] (2.550,1.530) rectangle (2.890,1.870);
\fill[pc3] (2.890,-0.850) rectangle (3.230,-0.510);
\fill[pc3] (3.230,-0.850) rectangle (3.570,-0.510);
\fill[pc3] (2.550,-0.850) rectangle (2.890,-0.510);
\fill[pc3] (2.890,-0.510) rectangle (3.230,-0.170);
\fill[pc3] (2.890,-1.190) rectangle (3.230,-0.850);
\fill[pc1] (2.890,0.850) rectangle (3.230,1.190);
\fill[pc1] (3.230,0.850) rectangle (3.570,1.190);
\fill[pc1] (2.550,0.850) rectangle (2.890,1.190);
\fill[pc1] (2.890,1.190) rectangle (3.230,1.530);
\fill[pc1] (2.890,0.510) rectangle (3.230,0.850);
\fill[pc4] (2.890,2.550) rectangle (3.230,2.890);
\fill[pc4] (3.230,2.550) rectangle (3.570,2.890);
\fill[pc4] (2.550,2.550) rectangle (2.890,2.890);
\fill[pc4] (2.890,2.890) rectangle (3.230,3.230);
\fill[pc4] (2.890,2.210) rectangle (3.230,2.550);
\fill[pc3] (3.230,-0.170) rectangle (3.570,0.170);
\fill[pc3] (3.570,-0.170) rectangle (3.910,0.170);
\fill[pc3] (2.890,-0.170) rectangle (3.230,0.170);
\fill[pc3] (3.230,0.170) rectangle (3.570,0.510);
\fill[pc3] (3.230,-0.510) rectangle (3.570,-0.170);
\fill[pc1] (3.230,1.530) rectangle (3.570,1.870);
\fill[pc1] (3.570,1.530) rectangle (3.910,1.870);
\fill[pc1] (2.890,1.530) rectangle (3.230,1.870);
\fill[pc1] (3.230,1.870) rectangle (3.570,2.210);
\fill[pc1] (3.230,1.190) rectangle (3.570,1.530);
\fill[pc4] (3.230,3.230) rectangle (3.570,3.570);
\fill[pc4] (3.570,3.230) rectangle (3.910,3.570);
\fill[pc4] (2.890,3.230) rectangle (3.230,3.570);
\fill[pc4] (3.230,3.570) rectangle (3.570,3.910);
\fill[pc4] (3.230,2.890) rectangle (3.570,3.230);
\fill[pc2] (3.570,0.510) rectangle (3.910,0.850);
\fill[pc2] (3.910,0.510) rectangle (4.250,0.850);
\fill[pc2] (3.230,0.510) rectangle (3.570,0.850);
\fill[pc2] (3.570,0.850) rectangle (3.910,1.190);
\fill[pc2] (3.570,0.170) rectangle (3.910,0.510);
\fill[pc5] (3.570,2.210) rectangle (3.910,2.550);
\fill[pc5] (3.910,2.210) rectangle (4.250,2.550);
\fill[pc5] (3.230,2.210) rectangle (3.570,2.550);
\fill[pc5] (3.570,2.550) rectangle (3.910,2.890);
\fill[pc5] (3.570,1.870) rectangle (3.910,2.210);
\fill[pc3] (3.910,-0.510) rectangle (4.250,-0.170);
\fill[pc3] (4.250,-0.510) rectangle (4.590,-0.170);
\fill[pc3] (3.570,-0.510) rectangle (3.910,-0.170);
\fill[pc3] (3.910,-0.170) rectangle (4.250,0.170);
\fill[pc3] (3.910,-0.850) rectangle (4.250,-0.510);
\fill[pc1] (3.910,1.190) rectangle (4.250,1.530);
\fill[pc1] (4.250,1.190) rectangle (4.590,1.530);
\fill[pc1] (3.570,1.190) rectangle (3.910,1.530);
\fill[pc1] (3.910,1.530) rectangle (4.250,1.870);
\fill[pc1] (3.910,0.850) rectangle (4.250,1.190);
\fill[pc4] (3.910,2.890) rectangle (4.250,3.230);
\fill[pc4] (4.250,2.890) rectangle (4.590,3.230);
\fill[pc4] (3.570,2.890) rectangle (3.910,3.230);
\fill[pc4] (3.910,3.230) rectangle (4.250,3.570);
\fill[pc4] (3.910,2.550) rectangle (4.250,2.890);
\fill[pc3] (4.250,0.170) rectangle (4.590,0.510);
\fill[pc3] (4.590,0.170) rectangle (4.930,0.510);
\fill[pc3] (3.910,0.170) rectangle (4.250,0.510);
\fill[pc3] (4.250,0.510) rectangle (4.590,0.850);
\fill[pc3] (4.250,-0.170) rectangle (4.590,0.170);
\fill[pc1] (4.250,1.870) rectangle (4.590,2.210);
\fill[pc1] (4.590,1.870) rectangle (4.930,2.210);
\fill[pc1] (3.910,1.870) rectangle (4.250,2.210);
\fill[pc1] (4.250,2.210) rectangle (4.590,2.550);
\fill[pc1] (4.250,1.530) rectangle (4.590,1.870);
\fill[pc4] (4.590,-0.850) rectangle (4.930,-0.510);
\fill[pc4] (4.930,-0.850) rectangle (5.270,-0.510);
\fill[pc4] (4.250,-0.850) rectangle (4.590,-0.510);
\fill[pc4] (4.590,-0.510) rectangle (4.930,-0.170);
\fill[pc4] (4.590,-1.190) rectangle (4.930,-0.850);
\fill[pc2] (4.590,0.850) rectangle (4.930,1.190);
\fill[pc2] (4.930,0.850) rectangle (5.270,1.190);
\fill[pc2] (4.250,0.850) rectangle (4.590,1.190);
\fill[pc2] (4.590,1.190) rectangle (4.930,1.530);
\fill[pc2] (4.590,0.510) rectangle (4.930,0.850);
\fill[pc5] (4.590,2.550) rectangle (4.930,2.890);
\fill[pc5] (4.930,2.550) rectangle (5.270,2.890);
\fill[pc5] (4.250,2.550) rectangle (4.590,2.890);
\fill[pc5] (4.590,2.890) rectangle (4.930,3.230);
\fill[pc5] (4.590,2.210) rectangle (4.930,2.550);
\fill[pc4] (4.930,-0.170) rectangle (5.270,0.170);
\fill[pc4] (5.270,-0.170) rectangle (5.610,0.170);
\fill[pc4] (4.590,-0.170) rectangle (4.930,0.170);
\fill[pc4] (4.930,0.170) rectangle (5.270,0.510);
\fill[pc4] (4.930,-0.510) rectangle (5.270,-0.170);
\fill[pc2] (4.930,1.530) rectangle (5.270,1.870);
\fill[pc2] (5.270,1.530) rectangle (5.610,1.870);
\fill[pc2] (4.590,1.530) rectangle (4.930,1.870);
\fill[pc2] (4.930,1.870) rectangle (5.270,2.210);
\fill[pc2] (4.930,1.190) rectangle (5.270,1.530);
\fill[pc5] (4.930,3.230) rectangle (5.270,3.570);
\fill[pc5] (5.270,3.230) rectangle (5.610,3.570);
\fill[pc5] (4.590,3.230) rectangle (4.930,3.570);
\fill[pc5] (4.930,3.570) rectangle (5.270,3.910);
\fill[pc5] (4.930,2.890) rectangle (5.270,3.230);
\fill[pc3] (5.270,0.510) rectangle (5.610,0.850);
\fill[pc3] (5.610,0.510) rectangle (5.950,0.850);
\fill[pc3] (4.930,0.510) rectangle (5.270,0.850);
\fill[pc3] (5.270,0.850) rectangle (5.610,1.190);
\fill[pc3] (5.270,0.170) rectangle (5.610,0.510);
\fill[pc1] (5.270,2.210) rectangle (5.610,2.550);
\fill[pc1] (5.610,2.210) rectangle (5.950,2.550);
\fill[pc1] (4.930,2.210) rectangle (5.270,2.550);
\fill[pc1] (5.270,2.550) rectangle (5.610,2.890);
\fill[pc1] (5.270,1.870) rectangle (5.610,2.210);
\draw[white,line width=0.4pt] (-0.510,-0.340)--(-0.510,3.060);
\draw[white,line width=0.4pt] (-0.170,-0.340)--(-0.170,3.060);
\draw[white,line width=0.4pt] (0.170,-0.340)--(0.170,3.060);
\draw[white,line width=0.4pt] (0.510,-0.340)--(0.510,3.060);
\draw[white,line width=0.4pt] (0.850,-0.340)--(0.850,3.060);
\draw[white,line width=0.4pt] (1.190,-0.340)--(1.190,3.060);
\draw[white,line width=0.4pt] (1.530,-0.340)--(1.530,3.060);
\draw[white,line width=0.4pt] (1.870,-0.340)--(1.870,3.060);
\draw[white,line width=0.4pt] (2.210,-0.340)--(2.210,3.060);
\draw[white,line width=0.4pt] (2.550,-0.340)--(2.550,3.060);
\draw[white,line width=0.4pt] (2.890,-0.340)--(2.890,3.060);
\draw[white,line width=0.4pt] (3.230,-0.340)--(3.230,3.060);
\draw[white,line width=0.4pt] (3.570,-0.340)--(3.570,3.060);
\draw[white,line width=0.4pt] (3.910,-0.340)--(3.910,3.060);
\draw[white,line width=0.4pt] (4.250,-0.340)--(4.250,3.060);
\draw[white,line width=0.4pt] (4.590,-0.340)--(4.590,3.060);
\draw[white,line width=0.4pt] (4.930,-0.340)--(4.930,3.060);
\draw[white,line width=0.4pt] (-0.340,-0.510)--(5.100,-0.510);
\draw[white,line width=0.4pt] (-0.340,-0.170)--(5.100,-0.170);
\draw[white,line width=0.4pt] (-0.340,0.170)--(5.100,0.170);
\draw[white,line width=0.4pt] (-0.340,0.510)--(5.100,0.510);
\draw[white,line width=0.4pt] (-0.340,0.850)--(5.100,0.850);
\draw[white,line width=0.4pt] (-0.340,1.190)--(5.100,1.190);
\draw[white,line width=0.4pt] (-0.340,1.530)--(5.100,1.530);
\draw[white,line width=0.4pt] (-0.340,1.870)--(5.100,1.870);
\draw[white,line width=0.4pt] (-0.340,2.210)--(5.100,2.210);
\draw[white,line width=0.4pt] (-0.340,2.550)--(5.100,2.550);
\draw[white,line width=0.4pt] (-0.340,2.890)--(5.100,2.890);
\fill[black] (-0.680,0.340) circle (0.05);
\fill[black] (-0.680,2.040) circle (0.05);
\fill[black] (-0.340,-0.680) circle (0.05);
\fill[black] (-0.340,1.020) circle (0.05);
\fill[black] (-0.340,2.720) circle (0.05);
\fill[black] (0.000,0.000) circle (0.05);
\fill[black] (0.000,1.700) circle (0.05);
\fill[black] (0.000,3.400) circle (0.05);
\fill[black] (0.340,0.680) circle (0.05);
\fill[black] (0.340,2.380) circle (0.05);
\fill[black] (0.680,-0.340) circle (0.05);
\fill[black] (0.680,1.360) circle (0.05);
\fill[black] (0.680,3.060) circle (0.05);
\fill[black] (1.020,0.340) circle (0.05);
\fill[black] (1.020,2.040) circle (0.05);
\fill[black] (1.360,-0.680) circle (0.05);
\fill[black] (1.360,1.020) circle (0.05);
\fill[black] (1.360,2.720) circle (0.05);
\fill[black] (1.700,0.000) circle (0.05);
\fill[black] (1.700,1.700) circle (0.05);
\fill[black] (1.700,3.400) circle (0.05);
\fill[black] (2.040,0.680) circle (0.05);
\fill[black] (2.040,2.380) circle (0.05);
\fill[black] (2.380,-0.340) circle (0.05);
\fill[black] (2.380,1.360) circle (0.05);
\fill[black] (2.380,3.060) circle (0.05);
\fill[black] (2.720,0.340) circle (0.05);
\fill[black] (2.720,2.040) circle (0.05);
\fill[black] (3.060,-0.680) circle (0.05);
\fill[black] (3.060,1.020) circle (0.05);
\fill[black] (3.060,2.720) circle (0.05);
\fill[black] (3.400,0.000) circle (0.05);
\fill[black] (3.400,1.700) circle (0.05);
\fill[black] (3.400,3.400) circle (0.05);
\fill[black] (3.740,0.680) circle (0.05);
\fill[black] (3.740,2.380) circle (0.05);
\fill[black] (4.080,-0.340) circle (0.05);
\fill[black] (4.080,1.360) circle (0.05);
\fill[black] (4.080,3.060) circle (0.05);
\fill[black] (4.420,0.340) circle (0.05);
\fill[black] (4.420,2.040) circle (0.05);
\fill[black] (4.760,-0.680) circle (0.05);
\fill[black] (4.760,1.020) circle (0.05);
\fill[black] (4.760,2.720) circle (0.05);
\fill[black] (5.100,0.000) circle (0.05);
\fill[black] (5.100,1.700) circle (0.05);
\fill[black] (5.100,3.400) circle (0.05);
\fill[black] (5.440,0.680) circle (0.05);
\fill[black] (5.440,2.380) circle (0.05);
\end{tikzpicture}
\caption{The perfect code $\{u+2v\equiv 0\pmod 5\}$ of $\Z^2$, drawn in the coordinates $(u,v)$: the $\ell_1$-balls of radius $1$ centred at its points tile the plane. In the coordinates $(s,g)=(u+v,\,v-u)$ it is the set $L_0$.}
\label{fig:perfect}
\end{figure}

\begin{remark}
Since $g(\sigma z)+3s(\sigma z)=4n+3s(z)-g(z)$, the glide reflection $\sigma$ maps each $L_c$ to a code of the mirror family $\{3s-g\equiv c'\pmod{10}\}$. Hence no set $L_c$ is $\sigma$-invariant, and a packing set built from $L_c$ must have a defect along a seam crossing the Möbius band. This is the source of the loss of about $n/10$ vertices with respect to $\frac15\binom n2$.
\end{remark}


\section{Proof of Theorem~\ref{maintheorem}}\label{sec:proof}

\subsection{The sets \texorpdfstring{$S_0$}{S0} and \texorpdfstring{$S_1$}{S1}}\label{sec:S}

For $\varepsilon\in\{0,1\}$ let
\begin{equation}\label{eq:S}
S_\varepsilon=\bigl\{z\in L_{2n+2\varepsilon}:\ \varepsilon\le s(z)\le n-2+\varepsilon,\ \ 1\le g(z)\le n-1\bigr\}\subseteq\Sigma .
\end{equation}
In other words, $S_\varepsilon$ consists of the points of a perfect code lying in the $n-1$ consecutive columns $\varepsilon\le s\le n-2+\varepsilon$ of the strip. These columns lie in $D$, so by Lemma~\ref{lem:model}(i) the map $\pi$ is injective on $S_\varepsilon$, and we identify $S_\varepsilon$ with $\pi(S_\varepsilon)$. By Lemma~\ref{lem:model}(i), in terms of pairs $0\le a<b\le n-1$ (congruences modulo $5$):
\[
\{a,b\}\in S_\varepsilon\iff
\begin{cases}
a+b\le n-2+\varepsilon\ \text{ and }\ a+2b\equiv n+\varepsilon,\quad\text{or}\\[1mm]
a+b\ge n+\varepsilon\ \text{ and }\ 2a+b\equiv 2n+\varepsilon .
\end{cases}
\]

\begin{example}\label{ex:19}
For $n=19$ and $\varepsilon=1$, the set $S_1$ consists of the $33$ pairs
$\{0,5\}$, $\{0,10\}$, $\{0,15\}$, $\{1,2\}$, $\{1,7\}$, $\{1,12\}$, $\{1,17\}$, $\{2,4\}$, $\{2,9\}$, $\{2,14\}$, $\{3,6\}$, $\{3,11\}$, $\{3,18\}$, $\{4,8\}$, $\{4,13\}$, $\{4,16\}$, $\{5,10\}$, $\{6,7\}$, $\{6,12\}$, $\{6,17\}$, $\{7,9\}$, $\{7,15\}$, $\{8,13\}$, $\{8,18\}$, $\{9,11\}$, $\{9,16\}$, $\{10,14\}$, $\{11,12\}$, $\{11,17\}$, $\{12,15\}$, $\{13,18\}$, $\{14,16\}$, $\{16,17\}$;
see Figure~\ref{fig:S19}. Here $\lceil 19\cdot 17/10\rceil=33$.
\end{example}

\begin{figure}[ht]
\centering
\begin{tikzpicture}
\draw[gridc] (0.000,0.500)--(0.250,0.750);
\draw[gridc] (0.000,0.500)--(0.250,0.250);
\draw[gridc] (0.000,1.000)--(0.250,1.250);
\draw[gridc] (0.000,1.000)--(0.250,0.750);
\draw[gridc] (0.000,1.500)--(0.250,1.750);
\draw[gridc] (0.000,1.500)--(0.250,1.250);
\draw[gridc] (0.000,2.000)--(0.250,2.250);
\draw[gridc] (0.000,2.000)--(0.250,1.750);
\draw[gridc] (0.000,2.500)--(0.250,2.750);
\draw[gridc] (0.000,2.500)--(0.250,2.250);
\draw[gridc] (0.000,3.000)--(0.250,3.250);
\draw[gridc] (0.000,3.000)--(0.250,2.750);
\draw[gridc] (0.000,3.500)--(0.250,3.750);
\draw[gridc] (0.000,3.500)--(0.250,3.250);
\draw[gridc] (0.000,4.000)--(0.250,4.250);
\draw[gridc] (0.000,4.000)--(0.250,3.750);
\draw[gridc] (0.000,4.500)--(0.250,4.250);
\draw[gridc] (0.250,0.250)--(0.500,0.500);
\draw[gridc] (0.250,0.750)--(0.500,1.000);
\draw[gridc] (0.250,0.750)--(0.500,0.500);
\draw[gridc] (0.250,1.250)--(0.500,1.500);
\draw[gridc] (0.250,1.250)--(0.500,1.000);
\draw[gridc] (0.250,1.750)--(0.500,2.000);
\draw[gridc] (0.250,1.750)--(0.500,1.500);
\draw[gridc] (0.250,2.250)--(0.500,2.500);
\draw[gridc] (0.250,2.250)--(0.500,2.000);
\draw[gridc] (0.250,2.750)--(0.500,3.000);
\draw[gridc] (0.250,2.750)--(0.500,2.500);
\draw[gridc] (0.250,3.250)--(0.500,3.500);
\draw[gridc] (0.250,3.250)--(0.500,3.000);
\draw[gridc] (0.250,3.750)--(0.500,4.000);
\draw[gridc] (0.250,3.750)--(0.500,3.500);
\draw[gridc] (0.250,4.250)--(0.500,4.500);
\draw[gridc] (0.250,4.250)--(0.500,4.000);
\draw[gridc] (0.500,0.500)--(0.750,0.750);
\draw[gridc] (0.500,0.500)--(0.750,0.250);
\draw[gridc] (0.500,1.000)--(0.750,1.250);
\draw[gridc] (0.500,1.000)--(0.750,0.750);
\draw[gridc] (0.500,1.500)--(0.750,1.750);
\draw[gridc] (0.500,1.500)--(0.750,1.250);
\draw[gridc] (0.500,2.000)--(0.750,2.250);
\draw[gridc] (0.500,2.000)--(0.750,1.750);
\draw[gridc] (0.500,2.500)--(0.750,2.750);
\draw[gridc] (0.500,2.500)--(0.750,2.250);
\draw[gridc] (0.500,3.000)--(0.750,3.250);
\draw[gridc] (0.500,3.000)--(0.750,2.750);
\draw[gridc] (0.500,3.500)--(0.750,3.750);
\draw[gridc] (0.500,3.500)--(0.750,3.250);
\draw[gridc] (0.500,4.000)--(0.750,4.250);
\draw[gridc] (0.500,4.000)--(0.750,3.750);
\draw[gridc] (0.500,4.500)--(0.750,4.250);
\draw[gridc] (0.750,0.250)--(1.000,0.500);
\draw[gridc] (0.750,0.750)--(1.000,1.000);
\draw[gridc] (0.750,0.750)--(1.000,0.500);
\draw[gridc] (0.750,1.250)--(1.000,1.500);
\draw[gridc] (0.750,1.250)--(1.000,1.000);
\draw[gridc] (0.750,1.750)--(1.000,2.000);
\draw[gridc] (0.750,1.750)--(1.000,1.500);
\draw[gridc] (0.750,2.250)--(1.000,2.500);
\draw[gridc] (0.750,2.250)--(1.000,2.000);
\draw[gridc] (0.750,2.750)--(1.000,3.000);
\draw[gridc] (0.750,2.750)--(1.000,2.500);
\draw[gridc] (0.750,3.250)--(1.000,3.500);
\draw[gridc] (0.750,3.250)--(1.000,3.000);
\draw[gridc] (0.750,3.750)--(1.000,4.000);
\draw[gridc] (0.750,3.750)--(1.000,3.500);
\draw[gridc] (0.750,4.250)--(1.000,4.500);
\draw[gridc] (0.750,4.250)--(1.000,4.000);
\draw[gridc] (1.000,0.500)--(1.250,0.750);
\draw[gridc] (1.000,0.500)--(1.250,0.250);
\draw[gridc] (1.000,1.000)--(1.250,1.250);
\draw[gridc] (1.000,1.000)--(1.250,0.750);
\draw[gridc] (1.000,1.500)--(1.250,1.750);
\draw[gridc] (1.000,1.500)--(1.250,1.250);
\draw[gridc] (1.000,2.000)--(1.250,2.250);
\draw[gridc] (1.000,2.000)--(1.250,1.750);
\draw[gridc] (1.000,2.500)--(1.250,2.750);
\draw[gridc] (1.000,2.500)--(1.250,2.250);
\draw[gridc] (1.000,3.000)--(1.250,3.250);
\draw[gridc] (1.000,3.000)--(1.250,2.750);
\draw[gridc] (1.000,3.500)--(1.250,3.750);
\draw[gridc] (1.000,3.500)--(1.250,3.250);
\draw[gridc] (1.000,4.000)--(1.250,4.250);
\draw[gridc] (1.000,4.000)--(1.250,3.750);
\draw[gridc] (1.000,4.500)--(1.250,4.250);
\draw[gridc] (1.250,0.250)--(1.500,0.500);
\draw[gridc] (1.250,0.750)--(1.500,1.000);
\draw[gridc] (1.250,0.750)--(1.500,0.500);
\draw[gridc] (1.250,1.250)--(1.500,1.500);
\draw[gridc] (1.250,1.250)--(1.500,1.000);
\draw[gridc] (1.250,1.750)--(1.500,2.000);
\draw[gridc] (1.250,1.750)--(1.500,1.500);
\draw[gridc] (1.250,2.250)--(1.500,2.500);
\draw[gridc] (1.250,2.250)--(1.500,2.000);
\draw[gridc] (1.250,2.750)--(1.500,3.000);
\draw[gridc] (1.250,2.750)--(1.500,2.500);
\draw[gridc] (1.250,3.250)--(1.500,3.500);
\draw[gridc] (1.250,3.250)--(1.500,3.000);
\draw[gridc] (1.250,3.750)--(1.500,4.000);
\draw[gridc] (1.250,3.750)--(1.500,3.500);
\draw[gridc] (1.250,4.250)--(1.500,4.500);
\draw[gridc] (1.250,4.250)--(1.500,4.000);
\draw[gridc] (1.500,0.500)--(1.750,0.750);
\draw[gridc] (1.500,0.500)--(1.750,0.250);
\draw[gridc] (1.500,1.000)--(1.750,1.250);
\draw[gridc] (1.500,1.000)--(1.750,0.750);
\draw[gridc] (1.500,1.500)--(1.750,1.750);
\draw[gridc] (1.500,1.500)--(1.750,1.250);
\draw[gridc] (1.500,2.000)--(1.750,2.250);
\draw[gridc] (1.500,2.000)--(1.750,1.750);
\draw[gridc] (1.500,2.500)--(1.750,2.750);
\draw[gridc] (1.500,2.500)--(1.750,2.250);
\draw[gridc] (1.500,3.000)--(1.750,3.250);
\draw[gridc] (1.500,3.000)--(1.750,2.750);
\draw[gridc] (1.500,3.500)--(1.750,3.750);
\draw[gridc] (1.500,3.500)--(1.750,3.250);
\draw[gridc] (1.500,4.000)--(1.750,4.250);
\draw[gridc] (1.500,4.000)--(1.750,3.750);
\draw[gridc] (1.500,4.500)--(1.750,4.250);
\draw[gridc] (1.750,0.250)--(2.000,0.500);
\draw[gridc] (1.750,0.750)--(2.000,1.000);
\draw[gridc] (1.750,0.750)--(2.000,0.500);
\draw[gridc] (1.750,1.250)--(2.000,1.500);
\draw[gridc] (1.750,1.250)--(2.000,1.000);
\draw[gridc] (1.750,1.750)--(2.000,2.000);
\draw[gridc] (1.750,1.750)--(2.000,1.500);
\draw[gridc] (1.750,2.250)--(2.000,2.500);
\draw[gridc] (1.750,2.250)--(2.000,2.000);
\draw[gridc] (1.750,2.750)--(2.000,3.000);
\draw[gridc] (1.750,2.750)--(2.000,2.500);
\draw[gridc] (1.750,3.250)--(2.000,3.500);
\draw[gridc] (1.750,3.250)--(2.000,3.000);
\draw[gridc] (1.750,3.750)--(2.000,4.000);
\draw[gridc] (1.750,3.750)--(2.000,3.500);
\draw[gridc] (1.750,4.250)--(2.000,4.500);
\draw[gridc] (1.750,4.250)--(2.000,4.000);
\draw[gridc] (2.000,0.500)--(2.250,0.750);
\draw[gridc] (2.000,0.500)--(2.250,0.250);
\draw[gridc] (2.000,1.000)--(2.250,1.250);
\draw[gridc] (2.000,1.000)--(2.250,0.750);
\draw[gridc] (2.000,1.500)--(2.250,1.750);
\draw[gridc] (2.000,1.500)--(2.250,1.250);
\draw[gridc] (2.000,2.000)--(2.250,2.250);
\draw[gridc] (2.000,2.000)--(2.250,1.750);
\draw[gridc] (2.000,2.500)--(2.250,2.750);
\draw[gridc] (2.000,2.500)--(2.250,2.250);
\draw[gridc] (2.000,3.000)--(2.250,3.250);
\draw[gridc] (2.000,3.000)--(2.250,2.750);
\draw[gridc] (2.000,3.500)--(2.250,3.750);
\draw[gridc] (2.000,3.500)--(2.250,3.250);
\draw[gridc] (2.000,4.000)--(2.250,4.250);
\draw[gridc] (2.000,4.000)--(2.250,3.750);
\draw[gridc] (2.000,4.500)--(2.250,4.250);
\draw[gridc] (2.250,0.250)--(2.500,0.500);
\draw[gridc] (2.250,0.750)--(2.500,1.000);
\draw[gridc] (2.250,0.750)--(2.500,0.500);
\draw[gridc] (2.250,1.250)--(2.500,1.500);
\draw[gridc] (2.250,1.250)--(2.500,1.000);
\draw[gridc] (2.250,1.750)--(2.500,2.000);
\draw[gridc] (2.250,1.750)--(2.500,1.500);
\draw[gridc] (2.250,2.250)--(2.500,2.500);
\draw[gridc] (2.250,2.250)--(2.500,2.000);
\draw[gridc] (2.250,2.750)--(2.500,3.000);
\draw[gridc] (2.250,2.750)--(2.500,2.500);
\draw[gridc] (2.250,3.250)--(2.500,3.500);
\draw[gridc] (2.250,3.250)--(2.500,3.000);
\draw[gridc] (2.250,3.750)--(2.500,4.000);
\draw[gridc] (2.250,3.750)--(2.500,3.500);
\draw[gridc] (2.250,4.250)--(2.500,4.500);
\draw[gridc] (2.250,4.250)--(2.500,4.000);
\draw[gridc] (2.500,0.500)--(2.750,0.750);
\draw[gridc] (2.500,0.500)--(2.750,0.250);
\draw[gridc] (2.500,1.000)--(2.750,1.250);
\draw[gridc] (2.500,1.000)--(2.750,0.750);
\draw[gridc] (2.500,1.500)--(2.750,1.750);
\draw[gridc] (2.500,1.500)--(2.750,1.250);
\draw[gridc] (2.500,2.000)--(2.750,2.250);
\draw[gridc] (2.500,2.000)--(2.750,1.750);
\draw[gridc] (2.500,2.500)--(2.750,2.750);
\draw[gridc] (2.500,2.500)--(2.750,2.250);
\draw[gridc] (2.500,3.000)--(2.750,3.250);
\draw[gridc] (2.500,3.000)--(2.750,2.750);
\draw[gridc] (2.500,3.500)--(2.750,3.750);
\draw[gridc] (2.500,3.500)--(2.750,3.250);
\draw[gridc] (2.500,4.000)--(2.750,4.250);
\draw[gridc] (2.500,4.000)--(2.750,3.750);
\draw[gridc] (2.500,4.500)--(2.750,4.250);
\draw[gridc] (2.750,0.250)--(3.000,0.500);
\draw[gridc] (2.750,0.750)--(3.000,1.000);
\draw[gridc] (2.750,0.750)--(3.000,0.500);
\draw[gridc] (2.750,1.250)--(3.000,1.500);
\draw[gridc] (2.750,1.250)--(3.000,1.000);
\draw[gridc] (2.750,1.750)--(3.000,2.000);
\draw[gridc] (2.750,1.750)--(3.000,1.500);
\draw[gridc] (2.750,2.250)--(3.000,2.500);
\draw[gridc] (2.750,2.250)--(3.000,2.000);
\draw[gridc] (2.750,2.750)--(3.000,3.000);
\draw[gridc] (2.750,2.750)--(3.000,2.500);
\draw[gridc] (2.750,3.250)--(3.000,3.500);
\draw[gridc] (2.750,3.250)--(3.000,3.000);
\draw[gridc] (2.750,3.750)--(3.000,4.000);
\draw[gridc] (2.750,3.750)--(3.000,3.500);
\draw[gridc] (2.750,4.250)--(3.000,4.500);
\draw[gridc] (2.750,4.250)--(3.000,4.000);
\draw[gridc] (3.000,0.500)--(3.250,0.750);
\draw[gridc] (3.000,0.500)--(3.250,0.250);
\draw[gridc] (3.000,1.000)--(3.250,1.250);
\draw[gridc] (3.000,1.000)--(3.250,0.750);
\draw[gridc] (3.000,1.500)--(3.250,1.750);
\draw[gridc] (3.000,1.500)--(3.250,1.250);
\draw[gridc] (3.000,2.000)--(3.250,2.250);
\draw[gridc] (3.000,2.000)--(3.250,1.750);
\draw[gridc] (3.000,2.500)--(3.250,2.750);
\draw[gridc] (3.000,2.500)--(3.250,2.250);
\draw[gridc] (3.000,3.000)--(3.250,3.250);
\draw[gridc] (3.000,3.000)--(3.250,2.750);
\draw[gridc] (3.000,3.500)--(3.250,3.750);
\draw[gridc] (3.000,3.500)--(3.250,3.250);
\draw[gridc] (3.000,4.000)--(3.250,4.250);
\draw[gridc] (3.000,4.000)--(3.250,3.750);
\draw[gridc] (3.000,4.500)--(3.250,4.250);
\draw[gridc] (3.250,0.250)--(3.500,0.500);
\draw[gridc] (3.250,0.750)--(3.500,1.000);
\draw[gridc] (3.250,0.750)--(3.500,0.500);
\draw[gridc] (3.250,1.250)--(3.500,1.500);
\draw[gridc] (3.250,1.250)--(3.500,1.000);
\draw[gridc] (3.250,1.750)--(3.500,2.000);
\draw[gridc] (3.250,1.750)--(3.500,1.500);
\draw[gridc] (3.250,2.250)--(3.500,2.500);
\draw[gridc] (3.250,2.250)--(3.500,2.000);
\draw[gridc] (3.250,2.750)--(3.500,3.000);
\draw[gridc] (3.250,2.750)--(3.500,2.500);
\draw[gridc] (3.250,3.250)--(3.500,3.500);
\draw[gridc] (3.250,3.250)--(3.500,3.000);
\draw[gridc] (3.250,3.750)--(3.500,4.000);
\draw[gridc] (3.250,3.750)--(3.500,3.500);
\draw[gridc] (3.250,4.250)--(3.500,4.500);
\draw[gridc] (3.250,4.250)--(3.500,4.000);
\draw[gridc] (3.500,0.500)--(3.750,0.750);
\draw[gridc] (3.500,0.500)--(3.750,0.250);
\draw[gridc] (3.500,1.000)--(3.750,1.250);
\draw[gridc] (3.500,1.000)--(3.750,0.750);
\draw[gridc] (3.500,1.500)--(3.750,1.750);
\draw[gridc] (3.500,1.500)--(3.750,1.250);
\draw[gridc] (3.500,2.000)--(3.750,2.250);
\draw[gridc] (3.500,2.000)--(3.750,1.750);
\draw[gridc] (3.500,2.500)--(3.750,2.750);
\draw[gridc] (3.500,2.500)--(3.750,2.250);
\draw[gridc] (3.500,3.000)--(3.750,3.250);
\draw[gridc] (3.500,3.000)--(3.750,2.750);
\draw[gridc] (3.500,3.500)--(3.750,3.750);
\draw[gridc] (3.500,3.500)--(3.750,3.250);
\draw[gridc] (3.500,4.000)--(3.750,4.250);
\draw[gridc] (3.500,4.000)--(3.750,3.750);
\draw[gridc] (3.500,4.500)--(3.750,4.250);
\draw[gridc] (3.750,0.250)--(4.000,0.500);
\draw[gridc] (3.750,0.750)--(4.000,1.000);
\draw[gridc] (3.750,0.750)--(4.000,0.500);
\draw[gridc] (3.750,1.250)--(4.000,1.500);
\draw[gridc] (3.750,1.250)--(4.000,1.000);
\draw[gridc] (3.750,1.750)--(4.000,2.000);
\draw[gridc] (3.750,1.750)--(4.000,1.500);
\draw[gridc] (3.750,2.250)--(4.000,2.500);
\draw[gridc] (3.750,2.250)--(4.000,2.000);
\draw[gridc] (3.750,2.750)--(4.000,3.000);
\draw[gridc] (3.750,2.750)--(4.000,2.500);
\draw[gridc] (3.750,3.250)--(4.000,3.500);
\draw[gridc] (3.750,3.250)--(4.000,3.000);
\draw[gridc] (3.750,3.750)--(4.000,4.000);
\draw[gridc] (3.750,3.750)--(4.000,3.500);
\draw[gridc] (3.750,4.250)--(4.000,4.500);
\draw[gridc] (3.750,4.250)--(4.000,4.000);
\draw[gridc] (4.000,0.500)--(4.250,0.750);
\draw[gridc] (4.000,0.500)--(4.250,0.250);
\draw[gridc] (4.000,1.000)--(4.250,1.250);
\draw[gridc] (4.000,1.000)--(4.250,0.750);
\draw[gridc] (4.000,1.500)--(4.250,1.750);
\draw[gridc] (4.000,1.500)--(4.250,1.250);
\draw[gridc] (4.000,2.000)--(4.250,2.250);
\draw[gridc] (4.000,2.000)--(4.250,1.750);
\draw[gridc] (4.000,2.500)--(4.250,2.750);
\draw[gridc] (4.000,2.500)--(4.250,2.250);
\draw[gridc] (4.000,3.000)--(4.250,3.250);
\draw[gridc] (4.000,3.000)--(4.250,2.750);
\draw[gridc] (4.000,3.500)--(4.250,3.750);
\draw[gridc] (4.000,3.500)--(4.250,3.250);
\draw[gridc] (4.000,4.000)--(4.250,4.250);
\draw[gridc] (4.000,4.000)--(4.250,3.750);
\draw[gridc] (4.000,4.500)--(4.250,4.250);
\draw[gridc] (4.250,0.250)--(4.500,0.500);
\draw[gridc] (4.250,0.750)--(4.500,1.000);
\draw[gridc] (4.250,0.750)--(4.500,0.500);
\draw[gridc] (4.250,1.250)--(4.500,1.500);
\draw[gridc] (4.250,1.250)--(4.500,1.000);
\draw[gridc] (4.250,1.750)--(4.500,2.000);
\draw[gridc] (4.250,1.750)--(4.500,1.500);
\draw[gridc] (4.250,2.250)--(4.500,2.500);
\draw[gridc] (4.250,2.250)--(4.500,2.000);
\draw[gridc] (4.250,2.750)--(4.500,3.000);
\draw[gridc] (4.250,2.750)--(4.500,2.500);
\draw[gridc] (4.250,3.250)--(4.500,3.500);
\draw[gridc] (4.250,3.250)--(4.500,3.000);
\draw[gridc] (4.250,3.750)--(4.500,4.000);
\draw[gridc] (4.250,3.750)--(4.500,3.500);
\draw[gridc] (4.250,4.250)--(4.500,4.500);
\draw[gridc] (4.250,4.250)--(4.500,4.000);
\fill[vtx] (0.000,0.500) circle (0.028);
\fill[vtx] (0.000,1.000) circle (0.028);
\fill[vtx] (0.000,1.500) circle (0.028);
\fill[vtx] (0.000,2.000) circle (0.028);
\fill[vtx] (0.000,2.500) circle (0.028);
\fill[vtx] (0.000,3.000) circle (0.028);
\fill[vtx] (0.000,3.500) circle (0.028);
\fill[vtx] (0.000,4.000) circle (0.028);
\fill[vtx] (0.000,4.500) circle (0.028);
\fill[vtx] (0.250,0.250) circle (0.028);
\fill[vtx] (0.250,0.750) circle (0.028);
\fill[vtx] (0.250,1.250) circle (0.028);
\fill[vtx] (0.250,1.750) circle (0.028);
\fill[vtx] (0.250,2.250) circle (0.028);
\fill[vtx] (0.250,2.750) circle (0.028);
\fill[vtx] (0.250,3.250) circle (0.028);
\fill[vtx] (0.250,3.750) circle (0.028);
\fill[vtx] (0.250,4.250) circle (0.028);
\fill[vtx] (0.500,0.500) circle (0.028);
\fill[vtx] (0.500,1.000) circle (0.028);
\fill[vtx] (0.500,1.500) circle (0.028);
\fill[vtx] (0.500,2.000) circle (0.028);
\fill[vtx] (0.500,2.500) circle (0.028);
\fill[vtx] (0.500,3.000) circle (0.028);
\fill[vtx] (0.500,3.500) circle (0.028);
\fill[vtx] (0.500,4.000) circle (0.028);
\fill[vtx] (0.500,4.500) circle (0.028);
\fill[vtx] (0.750,0.250) circle (0.028);
\fill[vtx] (0.750,0.750) circle (0.028);
\fill[vtx] (0.750,1.250) circle (0.028);
\fill[vtx] (0.750,1.750) circle (0.028);
\fill[vtx] (0.750,2.250) circle (0.028);
\fill[vtx] (0.750,2.750) circle (0.028);
\fill[vtx] (0.750,3.250) circle (0.028);
\fill[vtx] (0.750,3.750) circle (0.028);
\fill[vtx] (0.750,4.250) circle (0.028);
\fill[vtx] (1.000,0.500) circle (0.028);
\fill[vtx] (1.000,1.000) circle (0.028);
\fill[vtx] (1.000,1.500) circle (0.028);
\fill[vtx] (1.000,2.000) circle (0.028);
\fill[vtx] (1.000,2.500) circle (0.028);
\fill[vtx] (1.000,3.000) circle (0.028);
\fill[vtx] (1.000,3.500) circle (0.028);
\fill[vtx] (1.000,4.000) circle (0.028);
\fill[vtx] (1.000,4.500) circle (0.028);
\fill[vtx] (1.250,0.250) circle (0.028);
\fill[vtx] (1.250,0.750) circle (0.028);
\fill[vtx] (1.250,1.250) circle (0.028);
\fill[vtx] (1.250,1.750) circle (0.028);
\fill[vtx] (1.250,2.250) circle (0.028);
\fill[vtx] (1.250,2.750) circle (0.028);
\fill[vtx] (1.250,3.250) circle (0.028);
\fill[vtx] (1.250,3.750) circle (0.028);
\fill[vtx] (1.250,4.250) circle (0.028);
\fill[vtx] (1.500,0.500) circle (0.028);
\fill[vtx] (1.500,1.000) circle (0.028);
\fill[vtx] (1.500,1.500) circle (0.028);
\fill[vtx] (1.500,2.000) circle (0.028);
\fill[vtx] (1.500,2.500) circle (0.028);
\fill[vtx] (1.500,3.000) circle (0.028);
\fill[vtx] (1.500,3.500) circle (0.028);
\fill[vtx] (1.500,4.000) circle (0.028);
\fill[vtx] (1.500,4.500) circle (0.028);
\fill[vtx] (1.750,0.250) circle (0.028);
\fill[vtx] (1.750,0.750) circle (0.028);
\fill[vtx] (1.750,1.250) circle (0.028);
\fill[vtx] (1.750,1.750) circle (0.028);
\fill[vtx] (1.750,2.250) circle (0.028);
\fill[vtx] (1.750,2.750) circle (0.028);
\fill[vtx] (1.750,3.250) circle (0.028);
\fill[vtx] (1.750,3.750) circle (0.028);
\fill[vtx] (1.750,4.250) circle (0.028);
\fill[vtx] (2.000,0.500) circle (0.028);
\fill[vtx] (2.000,1.000) circle (0.028);
\fill[vtx] (2.000,1.500) circle (0.028);
\fill[vtx] (2.000,2.000) circle (0.028);
\fill[vtx] (2.000,2.500) circle (0.028);
\fill[vtx] (2.000,3.000) circle (0.028);
\fill[vtx] (2.000,3.500) circle (0.028);
\fill[vtx] (2.000,4.000) circle (0.028);
\fill[vtx] (2.000,4.500) circle (0.028);
\fill[vtx] (2.250,0.250) circle (0.028);
\fill[vtx] (2.250,0.750) circle (0.028);
\fill[vtx] (2.250,1.250) circle (0.028);
\fill[vtx] (2.250,1.750) circle (0.028);
\fill[vtx] (2.250,2.250) circle (0.028);
\fill[vtx] (2.250,2.750) circle (0.028);
\fill[vtx] (2.250,3.250) circle (0.028);
\fill[vtx] (2.250,3.750) circle (0.028);
\fill[vtx] (2.250,4.250) circle (0.028);
\fill[vtx] (2.500,0.500) circle (0.028);
\fill[vtx] (2.500,1.000) circle (0.028);
\fill[vtx] (2.500,1.500) circle (0.028);
\fill[vtx] (2.500,2.000) circle (0.028);
\fill[vtx] (2.500,2.500) circle (0.028);
\fill[vtx] (2.500,3.000) circle (0.028);
\fill[vtx] (2.500,3.500) circle (0.028);
\fill[vtx] (2.500,4.000) circle (0.028);
\fill[vtx] (2.500,4.500) circle (0.028);
\fill[vtx] (2.750,0.250) circle (0.028);
\fill[vtx] (2.750,0.750) circle (0.028);
\fill[vtx] (2.750,1.250) circle (0.028);
\fill[vtx] (2.750,1.750) circle (0.028);
\fill[vtx] (2.750,2.250) circle (0.028);
\fill[vtx] (2.750,2.750) circle (0.028);
\fill[vtx] (2.750,3.250) circle (0.028);
\fill[vtx] (2.750,3.750) circle (0.028);
\fill[vtx] (2.750,4.250) circle (0.028);
\fill[vtx] (3.000,0.500) circle (0.028);
\fill[vtx] (3.000,1.000) circle (0.028);
\fill[vtx] (3.000,1.500) circle (0.028);
\fill[vtx] (3.000,2.000) circle (0.028);
\fill[vtx] (3.000,2.500) circle (0.028);
\fill[vtx] (3.000,3.000) circle (0.028);
\fill[vtx] (3.000,3.500) circle (0.028);
\fill[vtx] (3.000,4.000) circle (0.028);
\fill[vtx] (3.000,4.500) circle (0.028);
\fill[vtx] (3.250,0.250) circle (0.028);
\fill[vtx] (3.250,0.750) circle (0.028);
\fill[vtx] (3.250,1.250) circle (0.028);
\fill[vtx] (3.250,1.750) circle (0.028);
\fill[vtx] (3.250,2.250) circle (0.028);
\fill[vtx] (3.250,2.750) circle (0.028);
\fill[vtx] (3.250,3.250) circle (0.028);
\fill[vtx] (3.250,3.750) circle (0.028);
\fill[vtx] (3.250,4.250) circle (0.028);
\fill[vtx] (3.500,0.500) circle (0.028);
\fill[vtx] (3.500,1.000) circle (0.028);
\fill[vtx] (3.500,1.500) circle (0.028);
\fill[vtx] (3.500,2.000) circle (0.028);
\fill[vtx] (3.500,2.500) circle (0.028);
\fill[vtx] (3.500,3.000) circle (0.028);
\fill[vtx] (3.500,3.500) circle (0.028);
\fill[vtx] (3.500,4.000) circle (0.028);
\fill[vtx] (3.500,4.500) circle (0.028);
\fill[vtx] (3.750,0.250) circle (0.028);
\fill[vtx] (3.750,0.750) circle (0.028);
\fill[vtx] (3.750,1.250) circle (0.028);
\fill[vtx] (3.750,1.750) circle (0.028);
\fill[vtx] (3.750,2.250) circle (0.028);
\fill[vtx] (3.750,2.750) circle (0.028);
\fill[vtx] (3.750,3.250) circle (0.028);
\fill[vtx] (3.750,3.750) circle (0.028);
\fill[vtx] (3.750,4.250) circle (0.028);
\fill[vtx] (4.000,0.500) circle (0.028);
\fill[vtx] (4.000,1.000) circle (0.028);
\fill[vtx] (4.000,1.500) circle (0.028);
\fill[vtx] (4.000,2.000) circle (0.028);
\fill[vtx] (4.000,2.500) circle (0.028);
\fill[vtx] (4.000,3.000) circle (0.028);
\fill[vtx] (4.000,3.500) circle (0.028);
\fill[vtx] (4.000,4.000) circle (0.028);
\fill[vtx] (4.000,4.500) circle (0.028);
\fill[vtx] (4.250,0.250) circle (0.028);
\fill[vtx] (4.250,0.750) circle (0.028);
\fill[vtx] (4.250,1.250) circle (0.028);
\fill[vtx] (4.250,1.750) circle (0.028);
\fill[vtx] (4.250,2.250) circle (0.028);
\fill[vtx] (4.250,2.750) circle (0.028);
\fill[vtx] (4.250,3.250) circle (0.028);
\fill[vtx] (4.250,3.750) circle (0.028);
\fill[vtx] (4.250,4.250) circle (0.028);
\fill[vtx] (4.500,0.500) circle (0.028);
\fill[vtx] (4.500,1.000) circle (0.028);
\fill[vtx] (4.500,1.500) circle (0.028);
\fill[vtx] (4.500,2.000) circle (0.028);
\fill[vtx] (4.500,2.500) circle (0.028);
\fill[vtx] (4.500,3.000) circle (0.028);
\fill[vtx] (4.500,3.500) circle (0.028);
\fill[vtx] (4.500,4.000) circle (0.028);
\fill[vtx] (4.500,4.500) circle (0.028);
\fill[codec] (1.500,3.000) circle (0.075);
\fill[codec] (3.000,1.000) circle (0.075);
\fill[codec] (0.750,0.250) circle (0.075);
\fill[codec] (4.250,2.250) circle (0.075);
\fill[codec] (1.000,4.500) circle (0.075);
\fill[codec] (2.000,1.500) circle (0.075);
\fill[codec] (3.250,2.750) circle (0.075);
\fill[codec] (2.750,4.250) circle (0.075);
\fill[codec] (0.500,3.500) circle (0.075);
\fill[codec] (3.750,1.250) circle (0.075);
\fill[codec] (1.500,0.500) circle (0.075);
\fill[codec] (4.500,4.000) circle (0.075);
\fill[codec] (1.000,2.000) circle (0.075);
\fill[codec] (3.500,4.500) circle (0.075);
\fill[codec] (1.250,3.750) circle (0.075);
\fill[codec] (0.500,1.000) circle (0.075);
\fill[codec] (2.750,1.750) circle (0.075);
\fill[codec] (2.250,3.250) circle (0.075);
\fill[codec] (3.250,0.250) circle (0.075);
\fill[codec] (4.000,3.000) circle (0.075);
\fill[codec] (0.250,4.250) circle (0.075);
\fill[codec] (1.750,2.250) circle (0.075);
\fill[codec] (4.500,1.500) circle (0.075);
\fill[codec] (3.000,3.500) circle (0.075);
\fill[codec] (0.750,2.750) circle (0.075);
\fill[codec] (3.500,2.000) circle (0.075);
\fill[codec] (1.250,1.250) circle (0.075);
\fill[codec] (2.250,0.750) circle (0.075);
\fill[codec] (2.000,4.000) circle (0.075);
\fill[codec] (2.500,2.500) circle (0.075);
\fill[codec] (0.250,1.750) circle (0.075);
\fill[codec] (4.000,0.500) circle (0.075);
\fill[codec] (3.750,3.750) circle (0.075);
\draw[delc,thick] (0.000,2.500) circle (0.09);
\draw[delc,thick] (4.250,2.250) circle (0.12);
\draw[seam,dashed,thick] (-0.125,0)--(-0.125,4.750);
\draw[seam,dashed,thick] (4.625,0)--(4.625,4.750);
\draw[->] (0,-0.350)--(4.850,-0.350) node[right]{\scriptsize $s$};
\draw[->] (-0.400,0)--(-0.400,4.850) node[above]{\scriptsize $g$};
\node[below] at (0,-0.35) {\tiny $0$};
\node[below] at (4.500,-0.35) {\tiny $n-1$};
\node[left] at (-0.4,0.250) {\tiny $1$};
\node[left] at (-0.4,4.500) {\tiny $n-1$};
\end{tikzpicture}
\caption{The set $S_1$ for $n=19$ inside $D$, in the coordinates $(s,g)$ (blue). The dashed lines are the two sides of the seam, which are identified by $\sigma$. The two circles mark the only point of the lattice in the discarded column $s=0$ and the point of $S_1$ with which it would collide across the seam.}
\label{fig:S19}
\end{figure}
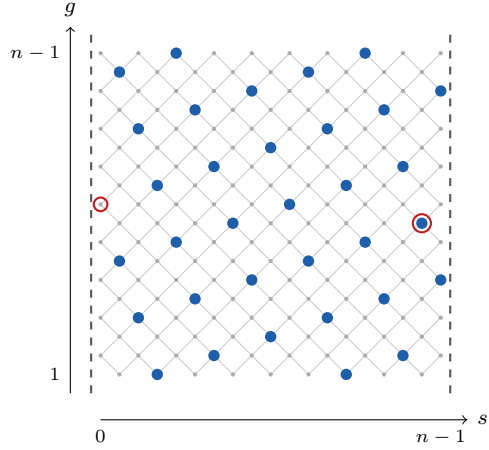

\begin{lemma}\label{lem:packing}
For $\varepsilon\in\{0,1\}$, $S_\varepsilon$ is a packing set of $F_2(C_n)$.
\end{lemma}

\begin{proof}
Let $x\neq y$ be points of $S_\varepsilon$. By Lemma~\ref{lem:model}(ii) it suffices to show that $\|x-\sigma^k(y)\|\ge 3$ for every $k\in\Z$. For $k=0$ this is Lemma~\ref{lem:code}. For $|k|\ge 2$ we have $|s(x)-s(\sigma^k y)|=|s(x)-s(y)-kn|\ge 2n-(n-2)>2$. Since $\|x-\sigma^{-1}(y)\|=\|y-\sigma(x)\|$, it remains to consider $k=1$ for every ordered pair $(x,y)$. Now
\[
s(\sigma y)-s(x)=n+s(y)-s(x)\ \ge\ n-(n-2)=2,
\]
with equality only if $s(y)=\varepsilon$ and $s(x)=n-2+\varepsilon$. In that case, reducing modulo $10$ the congruence defining $L_{2n+2\varepsilon}$,
\begin{align*}
g(x)&\equiv 2n+2\varepsilon-3(n-2+\varepsilon)=6-n-\varepsilon,\\
 g(y)&\equiv 2n+2\varepsilon-3\varepsilon=2n-\varepsilon,
\end{align*}
so that \[
g(x)-g(\sigma y)=g(x)+g(y)-n\equiv 6-2\varepsilon\pmod{10}.
\]

As $6-2\varepsilon\in\{4,6\}$, we get $|g(x)-g(\sigma y)|\ge 4$. In all cases $\|x-\sigma(y)\|\ge 3$.
\end{proof}

\begin{lemma}\label{lem:count}
Write $n-1=10q+t$ with $q\ge 0$ and $0\le t\le 9$. Then
\[
|S_\varepsilon|=q(10q+2t)+E_\varepsilon(t),
\]
where
\[
E_\varepsilon(t)=\#\bigl\{(i,j)\in\{1,\dots,t\}^2:\ i+3j\equiv 2t+5-\varepsilon\pmod{10}\bigr\}.
\]
\end{lemma}

\begin{proof}
Put $m=n-1$. The map $(s,g)\mapsto(i,j)=(g,\,s+1-\varepsilon)$ is a bijection from $S_\varepsilon$ onto the set of pairs $(i,j)\in\{1,\dots,m\}^2$ with $i+3j\equiv c\pmod{10}$, where $c=2n+3-\varepsilon\equiv 2t+5-\varepsilon$. Split $\{1,\dots,m\}=I\cup J$ with $I=\{1,\dots,10q\}$ and $J=\{10q+1,\dots,10q+t\}$. Since $I$ contains exactly $q$ elements of each residue class modulo $10$ and $3$ is invertible modulo $10$, there are exactly $qm$ admissible pairs with $i\in I$, and exactly $qt$ with $i\in J$ and $j\in I$. Finally, the admissible pairs with $i,j\in J$ correspond, after subtracting $10q$ from both coordinates, to those counted by $E_\varepsilon(t)$. Hence $|S_\varepsilon|=qm+qt+E_\varepsilon(t)$.
\end{proof}

\begin{table}[ht]
\centering
\begin{tabular}{lcccccccccc}
\toprule
$t$ & 0 & 1 & 2 & 3 & 4 & 5 & 6 & 7 & 8 & 9\\
\midrule
$E_0(t)$ & 0 & 0 & 0 & 1 & 2 & 2 & 4 & 5 & 6 & 8\\
$E_1(t)$ & 0 & 0 & 1 & 1 & 1 & 3 & 3 & 5 & 7 & 8\\
$\lceil (t^2-1)/10\rceil$ & 0 & 0 & 1 & 1 & 2 & 3 & 4 & 5 & 7 & 8\\
\bottomrule
\end{tabular}
\caption{The numbers $E_0(t)$ and $E_1(t)$ of Lemma~\ref{lem:count}.}
\label{tab:E}
\end{table}

\begin{corollary}\label{cor:count}
For every $n\ge 3$, $\max\{|S_0|,|S_1|\}=\lceil n(n-2)/10\rceil$. Moreover, if $n\equiv 0,2\pmod{10}$, then $|S_0|=n(n-2)/10$.
\end{corollary}

\begin{proof}
With $n-1=10q+t$ we have $n(n-2)=(n-1)^2-1=10q(10q+2t)+t^2-1$, hence
\[
\left\lceil\frac{n(n-2)}{10}\right\rceil=q(10q+2t)+\left\lceil\frac{t^2-1}{10}\right\rceil .
\]
The values of $E_\varepsilon(t)$, computed directly from the definition, are listed in Table~\ref{tab:E}; in every column $\max\{E_0(t),E_1(t)\}=\lceil(t^2-1)/10\rceil$, and $E_0(t)=\lceil(t^2-1)/10\rceil$ for $t\in\{1,9\}$, that is, for $n\equiv 2,0\pmod{10}$. The claims follow from Lemma~\ref{lem:count}.
\end{proof}

By Lemma~\ref{lem:packing} and Corollary~\ref{cor:count}, $\rho(F_2(C_n))\ge\lceil n(n-2)/10\rceil$ for every $n\ge 3$. If $n\not\equiv 0,2\pmod{10}$, then $n(n-2)/10\notin\Z$ and $\lceil n(n-2)/10\rceil=\lfloor n(n-2)/10\rfloor+1$, which proves Theorem~\ref{maintheorem} in this case.

\subsection{The case \texorpdfstring{$n\equiv 0,2\pmod{10}$}{n = 0,2 mod 10}}\label{sec:plus}

Assume now that $n\ge 10$ and $n\equiv 0,2\pmod{10}$. We gain one vertex by rearranging $S_0$ near the boundary of the Möbius band. Let $\beta=3$ if $n\equiv 0$ and $\beta=1$ if $n\equiv 2\pmod{10}$. Solving $g+3s\equiv 2n\pmod{10}$ for $s$ gives $s\equiv 4n+3g\pmod{10}$, that is, the points of $L_{2n}$ in the row $g$ are those with
\begin{equation}\label{eq:rows}
s\equiv \beta+3(g-1)\pmod{10}.
\end{equation}
Define $R=\{z\in S_0:\ g(z)\ge 4\}$,
\begin{align*}
B={}&\{(s,1)\colon -1\le s\le n-2,\ s\equiv \beta+2 \text{ or } \beta+8 \pmod{10}\}\\
    &\cup\{(s,2)\colon 0\le s\le n-2,\ s\equiv \beta+5 \pmod{10}\},
\end{align*}
and $S^+=R\cup B$. All points of $S^+$ have $0\le s\le n-2$, except the point $(-1,1)$, which belongs to $B$ when $\beta=1$; it represents the vertex $\{n-1,0\}$, and $\sigma(-1,1)=(n-1,n-1)$ lies in the column $s=n-1$, which contains no point of $S^+$. Hence, by Lemma~\ref{lem:model}(i), $\pi$ is injective on $S^+$, and again we identify $S^+$ with $\pi(S^+)$.

In terms of pairs (with indices modulo $n$), $S^+$ is obtained from $S_0$ as follows: each vertex of $S_0$ in the rows $g=1,2,3$, which has the form $\{a,a+1\}$, $\{a,a+2\}$ or $\{a,a+3\}$, is replaced by $\{a+1,a+2\}$, $\{a+1,a+3\}$ or $\{a+2,a+3\}$, respectively (a translation by $2$ in the coordinate $s$), and the vertex $\{0,1\}$ if $n\equiv 0$, or $\{n-1,0\}$ if $n\equiv 2\pmod{10}$, is added.

\begin{example}\label{ex:20}
For $n=20$ we have $\beta=3$, and $S^+$ consists of the $37$ pairs
$\{0,1\}$, $\{0,5\}$, $\{0,10\}$, $\{0,15\}$, $\{1,7\}$, $\{1,12\}$, $\{1,17\}$, $\{2,3\}$, $\{2,9\}$, $\{2,14\}$, $\{3,5\}$, $\{3,11\}$, $\{3,19\}$, $\{4,8\}$, $\{4,13\}$, $\{4,17\}$, $\{5,6\}$, $\{5,10\}$, $\{5,15\}$, $\{6,12\}$, $\{6,18\}$, $\{7,8\}$, $\{7,16\}$, $\{8,10\}$, $\{8,14\}$, $\{8,19\}$, $\{9,12\}$, $\{9,17\}$, $\{10,15\}$, $\{11,13\}$, $\{11,18\}$, $\{12,16\}$, $\{13,14\}$, $\{13,19\}$, $\{14,17\}$, $\{16,18\}$, $\{18,19\}$;
see Figure~\ref{fig:S20}. The set $R$ has $31$ points, the $5$ points of $S_0$ in the rows $g\le 3$ are discarded, and $B$ has $6$ points. Here $20\cdot 18/10+1=37$, while $a(20)=36$.
\end{example}

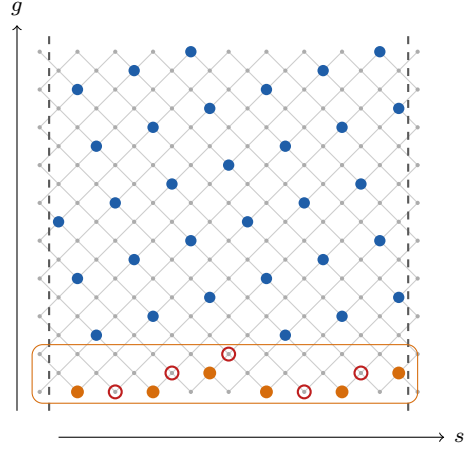
\begin{figure}[ht]
\centering
\begin{tikzpicture}
\draw[gridc] (-0.250,0.250)--(0.000,0.500);
\draw[gridc] (-0.250,0.750)--(0.000,1.000);
\draw[gridc] (-0.250,0.750)--(0.000,0.500);
\draw[gridc] (-0.250,1.250)--(0.000,1.500);
\draw[gridc] (-0.250,1.250)--(0.000,1.000);
\draw[gridc] (-0.250,1.750)--(0.000,2.000);
\draw[gridc] (-0.250,1.750)--(0.000,1.500);
\draw[gridc] (-0.250,2.250)--(0.000,2.500);
\draw[gridc] (-0.250,2.250)--(0.000,2.000);
\draw[gridc] (-0.250,2.750)--(0.000,3.000);
\draw[gridc] (-0.250,2.750)--(0.000,2.500);
\draw[gridc] (-0.250,3.250)--(0.000,3.500);
\draw[gridc] (-0.250,3.250)--(0.000,3.000);
\draw[gridc] (-0.250,3.750)--(0.000,4.000);
\draw[gridc] (-0.250,3.750)--(0.000,3.500);
\draw[gridc] (-0.250,4.250)--(0.000,4.500);
\draw[gridc] (-0.250,4.250)--(0.000,4.000);
\draw[gridc] (-0.250,4.750)--(0.000,4.500);
\draw[gridc] (0.000,0.500)--(0.250,0.750);
\draw[gridc] (0.000,0.500)--(0.250,0.250);
\draw[gridc] (0.000,1.000)--(0.250,1.250);
\draw[gridc] (0.000,1.000)--(0.250,0.750);
\draw[gridc] (0.000,1.500)--(0.250,1.750);
\draw[gridc] (0.000,1.500)--(0.250,1.250);
\draw[gridc] (0.000,2.000)--(0.250,2.250);
\draw[gridc] (0.000,2.000)--(0.250,1.750);
\draw[gridc] (0.000,2.500)--(0.250,2.750);
\draw[gridc] (0.000,2.500)--(0.250,2.250);
\draw[gridc] (0.000,3.000)--(0.250,3.250);
\draw[gridc] (0.000,3.000)--(0.250,2.750);
\draw[gridc] (0.000,3.500)--(0.250,3.750);
\draw[gridc] (0.000,3.500)--(0.250,3.250);
\draw[gridc] (0.000,4.000)--(0.250,4.250);
\draw[gridc] (0.000,4.000)--(0.250,3.750);
\draw[gridc] (0.000,4.500)--(0.250,4.750);
\draw[gridc] (0.000,4.500)--(0.250,4.250);
\draw[gridc] (0.250,0.250)--(0.500,0.500);
\draw[gridc] (0.250,0.750)--(0.500,1.000);
\draw[gridc] (0.250,0.750)--(0.500,0.500);
\draw[gridc] (0.250,1.250)--(0.500,1.500);
\draw[gridc] (0.250,1.250)--(0.500,1.000);
\draw[gridc] (0.250,1.750)--(0.500,2.000);
\draw[gridc] (0.250,1.750)--(0.500,1.500);
\draw[gridc] (0.250,2.250)--(0.500,2.500);
\draw[gridc] (0.250,2.250)--(0.500,2.000);
\draw[gridc] (0.250,2.750)--(0.500,3.000);
\draw[gridc] (0.250,2.750)--(0.500,2.500);
\draw[gridc] (0.250,3.250)--(0.500,3.500);
\draw[gridc] (0.250,3.250)--(0.500,3.000);
\draw[gridc] (0.250,3.750)--(0.500,4.000);
\draw[gridc] (0.250,3.750)--(0.500,3.500);
\draw[gridc] (0.250,4.250)--(0.500,4.500);
\draw[gridc] (0.250,4.250)--(0.500,4.000);
\draw[gridc] (0.250,4.750)--(0.500,4.500);
\draw[gridc] (0.500,0.500)--(0.750,0.750);
\draw[gridc] (0.500,0.500)--(0.750,0.250);
\draw[gridc] (0.500,1.000)--(0.750,1.250);
\draw[gridc] (0.500,1.000)--(0.750,0.750);
\draw[gridc] (0.500,1.500)--(0.750,1.750);
\draw[gridc] (0.500,1.500)--(0.750,1.250);
\draw[gridc] (0.500,2.000)--(0.750,2.250);
\draw[gridc] (0.500,2.000)--(0.750,1.750);
\draw[gridc] (0.500,2.500)--(0.750,2.750);
\draw[gridc] (0.500,2.500)--(0.750,2.250);
\draw[gridc] (0.500,3.000)--(0.750,3.250);
\draw[gridc] (0.500,3.000)--(0.750,2.750);
\draw[gridc] (0.500,3.500)--(0.750,3.750);
\draw[gridc] (0.500,3.500)--(0.750,3.250);
\draw[gridc] (0.500,4.000)--(0.750,4.250);
\draw[gridc] (0.500,4.000)--(0.750,3.750);
\draw[gridc] (0.500,4.500)--(0.750,4.750);
\draw[gridc] (0.500,4.500)--(0.750,4.250);
\draw[gridc] (0.750,0.250)--(1.000,0.500);
\draw[gridc] (0.750,0.750)--(1.000,1.000);
\draw[gridc] (0.750,0.750)--(1.000,0.500);
\draw[gridc] (0.750,1.250)--(1.000,1.500);
\draw[gridc] (0.750,1.250)--(1.000,1.000);
\draw[gridc] (0.750,1.750)--(1.000,2.000);
\draw[gridc] (0.750,1.750)--(1.000,1.500);
\draw[gridc] (0.750,2.250)--(1.000,2.500);
\draw[gridc] (0.750,2.250)--(1.000,2.000);
\draw[gridc] (0.750,2.750)--(1.000,3.000);
\draw[gridc] (0.750,2.750)--(1.000,2.500);
\draw[gridc] (0.750,3.250)--(1.000,3.500);
\draw[gridc] (0.750,3.250)--(1.000,3.000);
\draw[gridc] (0.750,3.750)--(1.000,4.000);
\draw[gridc] (0.750,3.750)--(1.000,3.500);
\draw[gridc] (0.750,4.250)--(1.000,4.500);
\draw[gridc] (0.750,4.250)--(1.000,4.000);
\draw[gridc] (0.750,4.750)--(1.000,4.500);
\draw[gridc] (1.000,0.500)--(1.250,0.750);
\draw[gridc] (1.000,0.500)--(1.250,0.250);
\draw[gridc] (1.000,1.000)--(1.250,1.250);
\draw[gridc] (1.000,1.000)--(1.250,0.750);
\draw[gridc] (1.000,1.500)--(1.250,1.750);
\draw[gridc] (1.000,1.500)--(1.250,1.250);
\draw[gridc] (1.000,2.000)--(1.250,2.250);
\draw[gridc] (1.000,2.000)--(1.250,1.750);
\draw[gridc] (1.000,2.500)--(1.250,2.750);
\draw[gridc] (1.000,2.500)--(1.250,2.250);
\draw[gridc] (1.000,3.000)--(1.250,3.250);
\draw[gridc] (1.000,3.000)--(1.250,2.750);
\draw[gridc] (1.000,3.500)--(1.250,3.750);
\draw[gridc] (1.000,3.500)--(1.250,3.250);
\draw[gridc] (1.000,4.000)--(1.250,4.250);
\draw[gridc] (1.000,4.000)--(1.250,3.750);
\draw[gridc] (1.000,4.500)--(1.250,4.750);
\draw[gridc] (1.000,4.500)--(1.250,4.250);
\draw[gridc] (1.250,0.250)--(1.500,0.500);
\draw[gridc] (1.250,0.750)--(1.500,1.000);
\draw[gridc] (1.250,0.750)--(1.500,0.500);
\draw[gridc] (1.250,1.250)--(1.500,1.500);
\draw[gridc] (1.250,1.250)--(1.500,1.000);
\draw[gridc] (1.250,1.750)--(1.500,2.000);
\draw[gridc] (1.250,1.750)--(1.500,1.500);
\draw[gridc] (1.250,2.250)--(1.500,2.500);
\draw[gridc] (1.250,2.250)--(1.500,2.000);
\draw[gridc] (1.250,2.750)--(1.500,3.000);
\draw[gridc] (1.250,2.750)--(1.500,2.500);
\draw[gridc] (1.250,3.250)--(1.500,3.500);
\draw[gridc] (1.250,3.250)--(1.500,3.000);
\draw[gridc] (1.250,3.750)--(1.500,4.000);
\draw[gridc] (1.250,3.750)--(1.500,3.500);
\draw[gridc] (1.250,4.250)--(1.500,4.500);
\draw[gridc] (1.250,4.250)--(1.500,4.000);
\draw[gridc] (1.250,4.750)--(1.500,4.500);
\draw[gridc] (1.500,0.500)--(1.750,0.750);
\draw[gridc] (1.500,0.500)--(1.750,0.250);
\draw[gridc] (1.500,1.000)--(1.750,1.250);
\draw[gridc] (1.500,1.000)--(1.750,0.750);
\draw[gridc] (1.500,1.500)--(1.750,1.750);
\draw[gridc] (1.500,1.500)--(1.750,1.250);
\draw[gridc] (1.500,2.000)--(1.750,2.250);
\draw[gridc] (1.500,2.000)--(1.750,1.750);
\draw[gridc] (1.500,2.500)--(1.750,2.750);
\draw[gridc] (1.500,2.500)--(1.750,2.250);
\draw[gridc] (1.500,3.000)--(1.750,3.250);
\draw[gridc] (1.500,3.000)--(1.750,2.750);
\draw[gridc] (1.500,3.500)--(1.750,3.750);
\draw[gridc] (1.500,3.500)--(1.750,3.250);
\draw[gridc] (1.500,4.000)--(1.750,4.250);
\draw[gridc] (1.500,4.000)--(1.750,3.750);
\draw[gridc] (1.500,4.500)--(1.750,4.750);
\draw[gridc] (1.500,4.500)--(1.750,4.250);
\draw[gridc] (1.750,0.250)--(2.000,0.500);
\draw[gridc] (1.750,0.750)--(2.000,1.000);
\draw[gridc] (1.750,0.750)--(2.000,0.500);
\draw[gridc] (1.750,1.250)--(2.000,1.500);
\draw[gridc] (1.750,1.250)--(2.000,1.000);
\draw[gridc] (1.750,1.750)--(2.000,2.000);
\draw[gridc] (1.750,1.750)--(2.000,1.500);
\draw[gridc] (1.750,2.250)--(2.000,2.500);
\draw[gridc] (1.750,2.250)--(2.000,2.000);
\draw[gridc] (1.750,2.750)--(2.000,3.000);
\draw[gridc] (1.750,2.750)--(2.000,2.500);
\draw[gridc] (1.750,3.250)--(2.000,3.500);
\draw[gridc] (1.750,3.250)--(2.000,3.000);
\draw[gridc] (1.750,3.750)--(2.000,4.000);
\draw[gridc] (1.750,3.750)--(2.000,3.500);
\draw[gridc] (1.750,4.250)--(2.000,4.500);
\draw[gridc] (1.750,4.250)--(2.000,4.000);
\draw[gridc] (1.750,4.750)--(2.000,4.500);
\draw[gridc] (2.000,0.500)--(2.250,0.750);
\draw[gridc] (2.000,0.500)--(2.250,0.250);
\draw[gridc] (2.000,1.000)--(2.250,1.250);
\draw[gridc] (2.000,1.000)--(2.250,0.750);
\draw[gridc] (2.000,1.500)--(2.250,1.750);
\draw[gridc] (2.000,1.500)--(2.250,1.250);
\draw[gridc] (2.000,2.000)--(2.250,2.250);
\draw[gridc] (2.000,2.000)--(2.250,1.750);
\draw[gridc] (2.000,2.500)--(2.250,2.750);
\draw[gridc] (2.000,2.500)--(2.250,2.250);
\draw[gridc] (2.000,3.000)--(2.250,3.250);
\draw[gridc] (2.000,3.000)--(2.250,2.750);
\draw[gridc] (2.000,3.500)--(2.250,3.750);
\draw[gridc] (2.000,3.500)--(2.250,3.250);
\draw[gridc] (2.000,4.000)--(2.250,4.250);
\draw[gridc] (2.000,4.000)--(2.250,3.750);
\draw[gridc] (2.000,4.500)--(2.250,4.750);
\draw[gridc] (2.000,4.500)--(2.250,4.250);
\draw[gridc] (2.250,0.250)--(2.500,0.500);
\draw[gridc] (2.250,0.750)--(2.500,1.000);
\draw[gridc] (2.250,0.750)--(2.500,0.500);
\draw[gridc] (2.250,1.250)--(2.500,1.500);
\draw[gridc] (2.250,1.250)--(2.500,1.000);
\draw[gridc] (2.250,1.750)--(2.500,2.000);
\draw[gridc] (2.250,1.750)--(2.500,1.500);
\draw[gridc] (2.250,2.250)--(2.500,2.500);
\draw[gridc] (2.250,2.250)--(2.500,2.000);
\draw[gridc] (2.250,2.750)--(2.500,3.000);
\draw[gridc] (2.250,2.750)--(2.500,2.500);
\draw[gridc] (2.250,3.250)--(2.500,3.500);
\draw[gridc] (2.250,3.250)--(2.500,3.000);
\draw[gridc] (2.250,3.750)--(2.500,4.000);
\draw[gridc] (2.250,3.750)--(2.500,3.500);
\draw[gridc] (2.250,4.250)--(2.500,4.500);
\draw[gridc] (2.250,4.250)--(2.500,4.000);
\draw[gridc] (2.250,4.750)--(2.500,4.500);
\draw[gridc] (2.500,0.500)--(2.750,0.750);
\draw[gridc] (2.500,0.500)--(2.750,0.250);
\draw[gridc] (2.500,1.000)--(2.750,1.250);
\draw[gridc] (2.500,1.000)--(2.750,0.750);
\draw[gridc] (2.500,1.500)--(2.750,1.750);
\draw[gridc] (2.500,1.500)--(2.750,1.250);
\draw[gridc] (2.500,2.000)--(2.750,2.250);
\draw[gridc] (2.500,2.000)--(2.750,1.750);
\draw[gridc] (2.500,2.500)--(2.750,2.750);
\draw[gridc] (2.500,2.500)--(2.750,2.250);
\draw[gridc] (2.500,3.000)--(2.750,3.250);
\draw[gridc] (2.500,3.000)--(2.750,2.750);
\draw[gridc] (2.500,3.500)--(2.750,3.750);
\draw[gridc] (2.500,3.500)--(2.750,3.250);
\draw[gridc] (2.500,4.000)--(2.750,4.250);
\draw[gridc] (2.500,4.000)--(2.750,3.750);
\draw[gridc] (2.500,4.500)--(2.750,4.750);
\draw[gridc] (2.500,4.500)--(2.750,4.250);
\draw[gridc] (2.750,0.250)--(3.000,0.500);
\draw[gridc] (2.750,0.750)--(3.000,1.000);
\draw[gridc] (2.750,0.750)--(3.000,0.500);
\draw[gridc] (2.750,1.250)--(3.000,1.500);
\draw[gridc] (2.750,1.250)--(3.000,1.000);
\draw[gridc] (2.750,1.750)--(3.000,2.000);
\draw[gridc] (2.750,1.750)--(3.000,1.500);
\draw[gridc] (2.750,2.250)--(3.000,2.500);
\draw[gridc] (2.750,2.250)--(3.000,2.000);
\draw[gridc] (2.750,2.750)--(3.000,3.000);
\draw[gridc] (2.750,2.750)--(3.000,2.500);
\draw[gridc] (2.750,3.250)--(3.000,3.500);
\draw[gridc] (2.750,3.250)--(3.000,3.000);
\draw[gridc] (2.750,3.750)--(3.000,4.000);
\draw[gridc] (2.750,3.750)--(3.000,3.500);
\draw[gridc] (2.750,4.250)--(3.000,4.500);
\draw[gridc] (2.750,4.250)--(3.000,4.000);
\draw[gridc] (2.750,4.750)--(3.000,4.500);
\draw[gridc] (3.000,0.500)--(3.250,0.750);
\draw[gridc] (3.000,0.500)--(3.250,0.250);
\draw[gridc] (3.000,1.000)--(3.250,1.250);
\draw[gridc] (3.000,1.000)--(3.250,0.750);
\draw[gridc] (3.000,1.500)--(3.250,1.750);
\draw[gridc] (3.000,1.500)--(3.250,1.250);
\draw[gridc] (3.000,2.000)--(3.250,2.250);
\draw[gridc] (3.000,2.000)--(3.250,1.750);
\draw[gridc] (3.000,2.500)--(3.250,2.750);
\draw[gridc] (3.000,2.500)--(3.250,2.250);
\draw[gridc] (3.000,3.000)--(3.250,3.250);
\draw[gridc] (3.000,3.000)--(3.250,2.750);
\draw[gridc] (3.000,3.500)--(3.250,3.750);
\draw[gridc] (3.000,3.500)--(3.250,3.250);
\draw[gridc] (3.000,4.000)--(3.250,4.250);
\draw[gridc] (3.000,4.000)--(3.250,3.750);
\draw[gridc] (3.000,4.500)--(3.250,4.750);
\draw[gridc] (3.000,4.500)--(3.250,4.250);
\draw[gridc] (3.250,0.250)--(3.500,0.500);
\draw[gridc] (3.250,0.750)--(3.500,1.000);
\draw[gridc] (3.250,0.750)--(3.500,0.500);
\draw[gridc] (3.250,1.250)--(3.500,1.500);
\draw[gridc] (3.250,1.250)--(3.500,1.000);
\draw[gridc] (3.250,1.750)--(3.500,2.000);
\draw[gridc] (3.250,1.750)--(3.500,1.500);
\draw[gridc] (3.250,2.250)--(3.500,2.500);
\draw[gridc] (3.250,2.250)--(3.500,2.000);
\draw[gridc] (3.250,2.750)--(3.500,3.000);
\draw[gridc] (3.250,2.750)--(3.500,2.500);
\draw[gridc] (3.250,3.250)--(3.500,3.500);
\draw[gridc] (3.250,3.250)--(3.500,3.000);
\draw[gridc] (3.250,3.750)--(3.500,4.000);
\draw[gridc] (3.250,3.750)--(3.500,3.500);
\draw[gridc] (3.250,4.250)--(3.500,4.500);
\draw[gridc] (3.250,4.250)--(3.500,4.000);
\draw[gridc] (3.250,4.750)--(3.500,4.500);
\draw[gridc] (3.500,0.500)--(3.750,0.750);
\draw[gridc] (3.500,0.500)--(3.750,0.250);
\draw[gridc] (3.500,1.000)--(3.750,1.250);
\draw[gridc] (3.500,1.000)--(3.750,0.750);
\draw[gridc] (3.500,1.500)--(3.750,1.750);
\draw[gridc] (3.500,1.500)--(3.750,1.250);
\draw[gridc] (3.500,2.000)--(3.750,2.250);
\draw[gridc] (3.500,2.000)--(3.750,1.750);
\draw[gridc] (3.500,2.500)--(3.750,2.750);
\draw[gridc] (3.500,2.500)--(3.750,2.250);
\draw[gridc] (3.500,3.000)--(3.750,3.250);
\draw[gridc] (3.500,3.000)--(3.750,2.750);
\draw[gridc] (3.500,3.500)--(3.750,3.750);
\draw[gridc] (3.500,3.500)--(3.750,3.250);
\draw[gridc] (3.500,4.000)--(3.750,4.250);
\draw[gridc] (3.500,4.000)--(3.750,3.750);
\draw[gridc] (3.500,4.500)--(3.750,4.750);
\draw[gridc] (3.500,4.500)--(3.750,4.250);
\draw[gridc] (3.750,0.250)--(4.000,0.500);
\draw[gridc] (3.750,0.750)--(4.000,1.000);
\draw[gridc] (3.750,0.750)--(4.000,0.500);
\draw[gridc] (3.750,1.250)--(4.000,1.500);
\draw[gridc] (3.750,1.250)--(4.000,1.000);
\draw[gridc] (3.750,1.750)--(4.000,2.000);
\draw[gridc] (3.750,1.750)--(4.000,1.500);
\draw[gridc] (3.750,2.250)--(4.000,2.500);
\draw[gridc] (3.750,2.250)--(4.000,2.000);
\draw[gridc] (3.750,2.750)--(4.000,3.000);
\draw[gridc] (3.750,2.750)--(4.000,2.500);
\draw[gridc] (3.750,3.250)--(4.000,3.500);
\draw[gridc] (3.750,3.250)--(4.000,3.000);
\draw[gridc] (3.750,3.750)--(4.000,4.000);
\draw[gridc] (3.750,3.750)--(4.000,3.500);
\draw[gridc] (3.750,4.250)--(4.000,4.500);
\draw[gridc] (3.750,4.250)--(4.000,4.000);
\draw[gridc] (3.750,4.750)--(4.000,4.500);
\draw[gridc] (4.000,0.500)--(4.250,0.750);
\draw[gridc] (4.000,0.500)--(4.250,0.250);
\draw[gridc] (4.000,1.000)--(4.250,1.250);
\draw[gridc] (4.000,1.000)--(4.250,0.750);
\draw[gridc] (4.000,1.500)--(4.250,1.750);
\draw[gridc] (4.000,1.500)--(4.250,1.250);
\draw[gridc] (4.000,2.000)--(4.250,2.250);
\draw[gridc] (4.000,2.000)--(4.250,1.750);
\draw[gridc] (4.000,2.500)--(4.250,2.750);
\draw[gridc] (4.000,2.500)--(4.250,2.250);
\draw[gridc] (4.000,3.000)--(4.250,3.250);
\draw[gridc] (4.000,3.000)--(4.250,2.750);
\draw[gridc] (4.000,3.500)--(4.250,3.750);
\draw[gridc] (4.000,3.500)--(4.250,3.250);
\draw[gridc] (4.000,4.000)--(4.250,4.250);
\draw[gridc] (4.000,4.000)--(4.250,3.750);
\draw[gridc] (4.000,4.500)--(4.250,4.750);
\draw[gridc] (4.000,4.500)--(4.250,4.250);
\draw[gridc] (4.250,0.250)--(4.500,0.500);
\draw[gridc] (4.250,0.750)--(4.500,1.000);
\draw[gridc] (4.250,0.750)--(4.500,0.500);
\draw[gridc] (4.250,1.250)--(4.500,1.500);
\draw[gridc] (4.250,1.250)--(4.500,1.000);
\draw[gridc] (4.250,1.750)--(4.500,2.000);
\draw[gridc] (4.250,1.750)--(4.500,1.500);
\draw[gridc] (4.250,2.250)--(4.500,2.500);
\draw[gridc] (4.250,2.250)--(4.500,2.000);
\draw[gridc] (4.250,2.750)--(4.500,3.000);
\draw[gridc] (4.250,2.750)--(4.500,2.500);
\draw[gridc] (4.250,3.250)--(4.500,3.500);
\draw[gridc] (4.250,3.250)--(4.500,3.000);
\draw[gridc] (4.250,3.750)--(4.500,4.000);
\draw[gridc] (4.250,3.750)--(4.500,3.500);
\draw[gridc] (4.250,4.250)--(4.500,4.500);
\draw[gridc] (4.250,4.250)--(4.500,4.000);
\draw[gridc] (4.250,4.750)--(4.500,4.500);
\draw[gridc] (4.500,0.500)--(4.750,0.750);
\draw[gridc] (4.500,0.500)--(4.750,0.250);
\draw[gridc] (4.500,1.000)--(4.750,1.250);
\draw[gridc] (4.500,1.000)--(4.750,0.750);
\draw[gridc] (4.500,1.500)--(4.750,1.750);
\draw[gridc] (4.500,1.500)--(4.750,1.250);
\draw[gridc] (4.500,2.000)--(4.750,2.250);
\draw[gridc] (4.500,2.000)--(4.750,1.750);
\draw[gridc] (4.500,2.500)--(4.750,2.750);
\draw[gridc] (4.500,2.500)--(4.750,2.250);
\draw[gridc] (4.500,3.000)--(4.750,3.250);
\draw[gridc] (4.500,3.000)--(4.750,2.750);
\draw[gridc] (4.500,3.500)--(4.750,3.750);
\draw[gridc] (4.500,3.500)--(4.750,3.250);
\draw[gridc] (4.500,4.000)--(4.750,4.250);
\draw[gridc] (4.500,4.000)--(4.750,3.750);
\draw[gridc] (4.500,4.500)--(4.750,4.750);
\draw[gridc] (4.500,4.500)--(4.750,4.250);
\fill[vtx] (-0.250,0.250) circle (0.028);
\fill[vtx] (-0.250,0.750) circle (0.028);
\fill[vtx] (-0.250,1.250) circle (0.028);
\fill[vtx] (-0.250,1.750) circle (0.028);
\fill[vtx] (-0.250,2.250) circle (0.028);
\fill[vtx] (-0.250,2.750) circle (0.028);
\fill[vtx] (-0.250,3.250) circle (0.028);
\fill[vtx] (-0.250,3.750) circle (0.028);
\fill[vtx] (-0.250,4.250) circle (0.028);
\fill[vtx] (-0.250,4.750) circle (0.028);
\fill[vtx] (0.000,0.500) circle (0.028);
\fill[vtx] (0.000,1.000) circle (0.028);
\fill[vtx] (0.000,1.500) circle (0.028);
\fill[vtx] (0.000,2.000) circle (0.028);
\fill[vtx] (0.000,2.500) circle (0.028);
\fill[vtx] (0.000,3.000) circle (0.028);
\fill[vtx] (0.000,3.500) circle (0.028);
\fill[vtx] (0.000,4.000) circle (0.028);
\fill[vtx] (0.000,4.500) circle (0.028);
\fill[vtx] (0.250,0.250) circle (0.028);
\fill[vtx] (0.250,0.750) circle (0.028);
\fill[vtx] (0.250,1.250) circle (0.028);
\fill[vtx] (0.250,1.750) circle (0.028);
\fill[vtx] (0.250,2.250) circle (0.028);
\fill[vtx] (0.250,2.750) circle (0.028);
\fill[vtx] (0.250,3.250) circle (0.028);
\fill[vtx] (0.250,3.750) circle (0.028);
\fill[vtx] (0.250,4.250) circle (0.028);
\fill[vtx] (0.250,4.750) circle (0.028);
\fill[vtx] (0.500,0.500) circle (0.028);
\fill[vtx] (0.500,1.000) circle (0.028);
\fill[vtx] (0.500,1.500) circle (0.028);
\fill[vtx] (0.500,2.000) circle (0.028);
\fill[vtx] (0.500,2.500) circle (0.028);
\fill[vtx] (0.500,3.000) circle (0.028);
\fill[vtx] (0.500,3.500) circle (0.028);
\fill[vtx] (0.500,4.000) circle (0.028);
\fill[vtx] (0.500,4.500) circle (0.028);
\fill[vtx] (0.750,0.250) circle (0.028);
\fill[vtx] (0.750,0.750) circle (0.028);
\fill[vtx] (0.750,1.250) circle (0.028);
\fill[vtx] (0.750,1.750) circle (0.028);
\fill[vtx] (0.750,2.250) circle (0.028);
\fill[vtx] (0.750,2.750) circle (0.028);
\fill[vtx] (0.750,3.250) circle (0.028);
\fill[vtx] (0.750,3.750) circle (0.028);
\fill[vtx] (0.750,4.250) circle (0.028);
\fill[vtx] (0.750,4.750) circle (0.028);
\fill[vtx] (1.000,0.500) circle (0.028);
\fill[vtx] (1.000,1.000) circle (0.028);
\fill[vtx] (1.000,1.500) circle (0.028);
\fill[vtx] (1.000,2.000) circle (0.028);
\fill[vtx] (1.000,2.500) circle (0.028);
\fill[vtx] (1.000,3.000) circle (0.028);
\fill[vtx] (1.000,3.500) circle (0.028);
\fill[vtx] (1.000,4.000) circle (0.028);
\fill[vtx] (1.000,4.500) circle (0.028);
\fill[vtx] (1.250,0.250) circle (0.028);
\fill[vtx] (1.250,0.750) circle (0.028);
\fill[vtx] (1.250,1.250) circle (0.028);
\fill[vtx] (1.250,1.750) circle (0.028);
\fill[vtx] (1.250,2.250) circle (0.028);
\fill[vtx] (1.250,2.750) circle (0.028);
\fill[vtx] (1.250,3.250) circle (0.028);
\fill[vtx] (1.250,3.750) circle (0.028);
\fill[vtx] (1.250,4.250) circle (0.028);
\fill[vtx] (1.250,4.750) circle (0.028);
\fill[vtx] (1.500,0.500) circle (0.028);
\fill[vtx] (1.500,1.000) circle (0.028);
\fill[vtx] (1.500,1.500) circle (0.028);
\fill[vtx] (1.500,2.000) circle (0.028);
\fill[vtx] (1.500,2.500) circle (0.028);
\fill[vtx] (1.500,3.000) circle (0.028);
\fill[vtx] (1.500,3.500) circle (0.028);
\fill[vtx] (1.500,4.000) circle (0.028);
\fill[vtx] (1.500,4.500) circle (0.028);
\fill[vtx] (1.750,0.250) circle (0.028);
\fill[vtx] (1.750,0.750) circle (0.028);
\fill[vtx] (1.750,1.250) circle (0.028);
\fill[vtx] (1.750,1.750) circle (0.028);
\fill[vtx] (1.750,2.250) circle (0.028);
\fill[vtx] (1.750,2.750) circle (0.028);
\fill[vtx] (1.750,3.250) circle (0.028);
\fill[vtx] (1.750,3.750) circle (0.028);
\fill[vtx] (1.750,4.250) circle (0.028);
\fill[vtx] (1.750,4.750) circle (0.028);
\fill[vtx] (2.000,0.500) circle (0.028);
\fill[vtx] (2.000,1.000) circle (0.028);
\fill[vtx] (2.000,1.500) circle (0.028);
\fill[vtx] (2.000,2.000) circle (0.028);
\fill[vtx] (2.000,2.500) circle (0.028);
\fill[vtx] (2.000,3.000) circle (0.028);
\fill[vtx] (2.000,3.500) circle (0.028);
\fill[vtx] (2.000,4.000) circle (0.028);
\fill[vtx] (2.000,4.500) circle (0.028);
\fill[vtx] (2.250,0.250) circle (0.028);
\fill[vtx] (2.250,0.750) circle (0.028);
\fill[vtx] (2.250,1.250) circle (0.028);
\fill[vtx] (2.250,1.750) circle (0.028);
\fill[vtx] (2.250,2.250) circle (0.028);
\fill[vtx] (2.250,2.750) circle (0.028);
\fill[vtx] (2.250,3.250) circle (0.028);
\fill[vtx] (2.250,3.750) circle (0.028);
\fill[vtx] (2.250,4.250) circle (0.028);
\fill[vtx] (2.250,4.750) circle (0.028);
\fill[vtx] (2.500,0.500) circle (0.028);
\fill[vtx] (2.500,1.000) circle (0.028);
\fill[vtx] (2.500,1.500) circle (0.028);
\fill[vtx] (2.500,2.000) circle (0.028);
\fill[vtx] (2.500,2.500) circle (0.028);
\fill[vtx] (2.500,3.000) circle (0.028);
\fill[vtx] (2.500,3.500) circle (0.028);
\fill[vtx] (2.500,4.000) circle (0.028);
\fill[vtx] (2.500,4.500) circle (0.028);
\fill[vtx] (2.750,0.250) circle (0.028);
\fill[vtx] (2.750,0.750) circle (0.028);
\fill[vtx] (2.750,1.250) circle (0.028);
\fill[vtx] (2.750,1.750) circle (0.028);
\fill[vtx] (2.750,2.250) circle (0.028);
\fill[vtx] (2.750,2.750) circle (0.028);
\fill[vtx] (2.750,3.250) circle (0.028);
\fill[vtx] (2.750,3.750) circle (0.028);
\fill[vtx] (2.750,4.250) circle (0.028);
\fill[vtx] (2.750,4.750) circle (0.028);
\fill[vtx] (3.000,0.500) circle (0.028);
\fill[vtx] (3.000,1.000) circle (0.028);
\fill[vtx] (3.000,1.500) circle (0.028);
\fill[vtx] (3.000,2.000) circle (0.028);
\fill[vtx] (3.000,2.500) circle (0.028);
\fill[vtx] (3.000,3.000) circle (0.028);
\fill[vtx] (3.000,3.500) circle (0.028);
\fill[vtx] (3.000,4.000) circle (0.028);
\fill[vtx] (3.000,4.500) circle (0.028);
\fill[vtx] (3.250,0.250) circle (0.028);
\fill[vtx] (3.250,0.750) circle (0.028);
\fill[vtx] (3.250,1.250) circle (0.028);
\fill[vtx] (3.250,1.750) circle (0.028);
\fill[vtx] (3.250,2.250) circle (0.028);
\fill[vtx] (3.250,2.750) circle (0.028);
\fill[vtx] (3.250,3.250) circle (0.028);
\fill[vtx] (3.250,3.750) circle (0.028);
\fill[vtx] (3.250,4.250) circle (0.028);
\fill[vtx] (3.250,4.750) circle (0.028);
\fill[vtx] (3.500,0.500) circle (0.028);
\fill[vtx] (3.500,1.000) circle (0.028);
\fill[vtx] (3.500,1.500) circle (0.028);
\fill[vtx] (3.500,2.000) circle (0.028);
\fill[vtx] (3.500,2.500) circle (0.028);
\fill[vtx] (3.500,3.000) circle (0.028);
\fill[vtx] (3.500,3.500) circle (0.028);
\fill[vtx] (3.500,4.000) circle (0.028);
\fill[vtx] (3.500,4.500) circle (0.028);
\fill[vtx] (3.750,0.250) circle (0.028);
\fill[vtx] (3.750,0.750) circle (0.028);
\fill[vtx] (3.750,1.250) circle (0.028);
\fill[vtx] (3.750,1.750) circle (0.028);
\fill[vtx] (3.750,2.250) circle (0.028);
\fill[vtx] (3.750,2.750) circle (0.028);
\fill[vtx] (3.750,3.250) circle (0.028);
\fill[vtx] (3.750,3.750) circle (0.028);
\fill[vtx] (3.750,4.250) circle (0.028);
\fill[vtx] (3.750,4.750) circle (0.028);
\fill[vtx] (4.000,0.500) circle (0.028);
\fill[vtx] (4.000,1.000) circle (0.028);
\fill[vtx] (4.000,1.500) circle (0.028);
\fill[vtx] (4.000,2.000) circle (0.028);
\fill[vtx] (4.000,2.500) circle (0.028);
\fill[vtx] (4.000,3.000) circle (0.028);
\fill[vtx] (4.000,3.500) circle (0.028);
\fill[vtx] (4.000,4.000) circle (0.028);
\fill[vtx] (4.000,4.500) circle (0.028);
\fill[vtx] (4.250,0.250) circle (0.028);
\fill[vtx] (4.250,0.750) circle (0.028);
\fill[vtx] (4.250,1.250) circle (0.028);
\fill[vtx] (4.250,1.750) circle (0.028);
\fill[vtx] (4.250,2.250) circle (0.028);
\fill[vtx] (4.250,2.750) circle (0.028);
\fill[vtx] (4.250,3.250) circle (0.028);
\fill[vtx] (4.250,3.750) circle (0.028);
\fill[vtx] (4.250,4.250) circle (0.028);
\fill[vtx] (4.250,4.750) circle (0.028);
\fill[vtx] (4.500,0.500) circle (0.028);
\fill[vtx] (4.500,1.000) circle (0.028);
\fill[vtx] (4.500,1.500) circle (0.028);
\fill[vtx] (4.500,2.000) circle (0.028);
\fill[vtx] (4.500,2.500) circle (0.028);
\fill[vtx] (4.500,3.000) circle (0.028);
\fill[vtx] (4.500,3.500) circle (0.028);
\fill[vtx] (4.500,4.000) circle (0.028);
\fill[vtx] (4.500,4.500) circle (0.028);
\fill[vtx] (4.750,0.250) circle (0.028);
\fill[vtx] (4.750,0.750) circle (0.028);
\fill[vtx] (4.750,1.250) circle (0.028);
\fill[vtx] (4.750,1.750) circle (0.028);
\fill[vtx] (4.750,2.250) circle (0.028);
\fill[vtx] (4.750,2.750) circle (0.028);
\fill[vtx] (4.750,3.250) circle (0.028);
\fill[vtx] (4.750,3.750) circle (0.028);
\fill[vtx] (4.750,4.250) circle (0.028);
\fill[vtx] (4.750,4.750) circle (0.028);
\fill[codec] (1.500,3.000) circle (0.075);
\fill[codec] (3.000,1.000) circle (0.075);
\fill[codec] (4.250,2.250) circle (0.075);
\fill[codec] (1.000,4.500) circle (0.075);
\fill[codec] (2.000,1.500) circle (0.075);
\fill[codec] (3.250,2.750) circle (0.075);
\fill[codec] (2.750,4.250) circle (0.075);
\fill[codec] (0.500,3.500) circle (0.075);
\fill[codec] (3.750,1.250) circle (0.075);
\fill[codec] (4.500,4.000) circle (0.075);
\fill[codec] (1.750,4.750) circle (0.075);
\fill[codec] (1.000,2.000) circle (0.075);
\fill[codec] (3.500,4.500) circle (0.075);
\fill[codec] (1.250,3.750) circle (0.075);
\fill[codec] (0.500,1.000) circle (0.075);
\fill[codec] (2.750,1.750) circle (0.075);
\fill[codec] (2.250,3.250) circle (0.075);
\fill[codec] (0.000,2.500) circle (0.075);
\fill[codec] (4.000,3.000) circle (0.075);
\fill[codec] (0.250,4.250) circle (0.075);
\fill[codec] (1.750,2.250) circle (0.075);
\fill[codec] (4.500,1.500) circle (0.075);
\fill[codec] (3.000,3.500) circle (0.075);
\fill[codec] (0.750,2.750) circle (0.075);
\fill[codec] (3.500,2.000) circle (0.075);
\fill[codec] (1.250,1.250) circle (0.075);
\fill[codec] (4.250,4.750) circle (0.075);
\fill[codec] (2.000,4.000) circle (0.075);
\fill[codec] (2.500,2.500) circle (0.075);
\fill[codec] (0.250,1.750) circle (0.075);
\fill[codec] (3.750,3.750) circle (0.075);
\draw[delc,thick] (0.750,0.250) circle (0.085);
\draw[delc,thick] (1.500,0.500) circle (0.085);
\draw[delc,thick] (3.250,0.250) circle (0.085);
\draw[delc,thick] (2.250,0.750) circle (0.085);
\draw[delc,thick] (4.000,0.500) circle (0.085);
\fill[newc] (0.250,0.250) circle (0.085);
\fill[newc] (1.250,0.250) circle (0.085);
\fill[newc] (2.750,0.250) circle (0.085);
\fill[newc] (3.750,0.250) circle (0.085);
\fill[newc] (2.000,0.500) circle (0.085);
\fill[newc] (4.500,0.500) circle (0.085);
\draw[seam,dashed,thick] (-0.125,0)--(-0.125,5.000);
\draw[seam,dashed,thick] (4.625,0)--(4.625,5.000);
\draw[->] (0,-0.350)--(5.100,-0.350) node[right]{\scriptsize $s$};
\draw[->] (-0.550,0)--(-0.550,5.100) node[above]{\scriptsize $g$};
\draw[newc,rounded corners] (-0.350,0.100) rectangle (4.750,0.875);
\end{tikzpicture}
\caption{The set $S^+$ for $n=20$ in the coordinates $(s,g)$: in blue the points of $R$, in orange the points of $B$ (inside the box), and in red the points of $S_0$ in the rows $g\le 3$, which are discarded. The dashed lines are the two sides of the seam.}
\label{fig:S20}
\end{figure}

\begin{lemma}\label{lem:pluscount}
$|S^+|=n(n-2)/10+1$.
\end{lemma}

\begin{proof}
By Corollary~\ref{cor:count} it suffices to show that $|B|=|S_0\setminus R|+1$. We count points in intervals with a prescribed residue modulo $10$, using \eqref{eq:rows}. If $n=10q$, then $\beta=3$ and the window is $0\le s\le 10q-2$: the rows $g=1,2,3$ of $S_0$ contain the points with $s\equiv 3,6,9$, that is, $q$, $q$ and $q-1$ points, while $B$ contains $q+q$ points in row $1$ (residues $5$ and $1$ in $[-1,10q-2]$) and $q$ points in row $2$ (residue $8$). Thus $|S_0\setminus R|=3q-1$ and $|B|=3q$. If $n=10q+2$, then $\beta=1$ and the window is $0\le s\le 10q$: the rows $g=1,2,3$ of $S_0$ contain $q$ points each (residues $1,4,7$), while $B$ contains $q$ points with $s\equiv 3$ and $q+1$ points with $s\equiv 9$ in row $1$ (among the latter, $s=-1$), and $q$ points in row $2$ (residue $6$). Thus $|S_0\setminus R|=3q$ and $|B|=3q+1$.
\end{proof}

\begin{lemma}\label{lem:pluspacking}
$S^+$ is a packing set of $F_2(C_n)$.
\end{lemma}

\begin{proof}
By Lemma~\ref{lem:model}(ii) it suffices to show that $\|x-\sigma^k(y)\|\ge 3$ for all distinct $x,y\in S^+$ and all $k\in\Z$.

\emph{Case $k=0$.} Two points of $R\subseteq L_{2n}$ are at distance at least $3$ by Lemma~\ref{lem:code}. For two points of $B$, the difference of their $s$-coordinates is congruent modulo $10$ to $0$ or $\pm 6$ if both lie in row $1$, to $\pm3$ if they lie in different rows, and to $0$ if both lie in row $2$; in all cases its absolute value is at least $3$. Finally, let $x\in R$ and $y\in B$. If $g(x)\ge 5$, or if $g(y)=1$, then $|g(x)-g(y)|\ge 3$. Otherwise $g(x)=4$ and $g(y)=2$, and by \eqref{eq:rows} we have $s(x)\equiv\beta+9$ and $s(y)\equiv\beta+5\pmod{10}$, so $|s(x)-s(y)|\ge 4$.

\emph{Case $|k|\ge 2$.} All points of $S^+$ satisfy $-1\le s\le n-2$, hence $|s(x)-s(\sigma^k y)|\ge 2n-(n-1)>2$.

\emph{Case $k=\pm1$.} As before it suffices to take $k=1$ for every ordered pair $(x,y)$. Now $s(\sigma y)-s(x)=n+s(y)-s(x)$ is at most $2$ only if
\[
(s(y),s(x))\in\{(0,n-2),\ (-1,n-3),\ (-1,n-2)\}.
\]
If $x,y\in R\subseteq S_0$ we are done by Lemma~\ref{lem:packing}, so we may assume that $x$ or $y$ lies in $B$. By \eqref{eq:rows} and the definition of $B$, the points of $R$ with $s=0$, $s=n-3$ and $s=n-2$ have $g\equiv 2n$, $g\equiv 9-n$ and $g\equiv 6-n\pmod{10}$, respectively, and the points of $B$ with $s\in\{-1,0,n-3,n-2\}$ are: $(n-2,2)$ if $n\equiv 0\pmod{10}$; $(-1,1)$ and $(n-3,1)$ if $n\equiv 2\pmod{10}$. Put $\delta=g(x)-g(\sigma y)=g(x)+g(y)-n$.

If $n\equiv 0\pmod{10}$, the only possibility is $x=(n-2,2)$ and $y=(0,g)\in R$ with $g\equiv 0\pmod{10}$; then $g\le n-10$ and $\delta=2+g-n\le -8$.

If $n\equiv 2\pmod{10}$, then $y=(-1,1)$, and $x$ is $(n-3,1)$ or a point of $R$ with $s(x)\in\{n-3,n-2\}$. If $x=(n-3,1)$, then $\delta=2-n\le -10$. If $x=(n-3,g)\in R$, then $g\equiv 9-n\equiv 7$, so $\delta=g+1-n\equiv 6\pmod{10}$ and $|\delta|\ge 4$. If $x=(n-2,g)\in R$, then $g\equiv 6-n\equiv 4$, so $\delta\equiv 3\pmod{10}$ and $|\delta|\ge 3$.

In all cases $\|x-\sigma(y)\|\ge|\delta|\ge 3$.
\end{proof}

\begin{proof}[Proof of Theorem~\ref{maintheorem}]
If $n\not\equiv 0,2\pmod{10}$, the theorem was proved at the end of Section~\ref{sec:S}. If $n\equiv 0,2\pmod{10}$, then $n\ge 10$ and the theorem follows from Lemmas~\ref{lem:pluscount} and~\ref{lem:pluspacking}.
\end{proof}

\section*{Acknowledgements}
It is a pleasure to thank J.~M.~Gómez Soto and L.~M.~Ríos-Castro. This note would not exist without their paper \cite{GR}: after some detailed reading of his paper, I was aware that the construction studied here is essentially theirs.\\

\section*{Use of generative AI}
The author's idea for a simplified proof was refined through a series of interactions with the AI assistant Claude (Anthropic). The AI was also used to revise the exposition and to write the code for the computer-assisted verifications. The author independently verified all mathematical arguments, computations, and references, and takes full responsibility for the contents of this note.


\end{document}